\documentclass{amsart}
\usepackage{hyperref}

\usepackage{latexsym,amssymb,amsmath,amsfonts,amstext,xcolor}

\usepackage{mathrsfs}
\usepackage[toc,page]{appendix}

\usepackage{enumerate}
\usepackage{graphicx, float}

\usepackage[margin=1.5in]{geometry}

\newtheorem{thm}{Theorem}[section]
\newtheorem{cor}[thm]{Corollary}
\newtheorem{lem}[thm]{Lemma}
\newtheorem{prop}[thm]{Proposition}
\newtheorem{con}[thm]{Conjecture}

\theoremstyle{definition}

\theoremstyle{definition}
\newtheorem{defi}[thm]{Definition}
\newtheorem{rem}[thm]{Remark}

\title[The unconstrained 100 prisoner problem]{On partial information retrieval: \\the unconstrained 100 prisoner problem}

\author{Ivano Lodato}
\email{ivano.lodato@gmail.com}
\address{Allos Limited, Hong Kong}

\author{Snehal M. Shekatkar}
\email{snehal@inferred.in}
\address{Department  of Scientific Computing, Modeling, and Simulation, Savitribai Phule Pune University, Pune 411007, India}

\author{Tian An Wong}
\email{tiananw@umich.edu}
\address{Department of Mathematics and Statistics, University of Michigan, Dearborn,  MI 48126 USA}


\begin{document}

\begin{abstract}
We consider a generalization of the classical 100 Prisoner problem and its variant, involving empty boxes, whereby winning probabilities for a team depend on the number of attempts, as well as on the number of winners. We call this the unconstrained 100 prisoner problem. After introducing the 3 main classes of strategies, we define a variety of `hybrid' strategies and quantify their winning-efficiency. Whenever analytic results are not available, we make use of Monte Carlo simulations to estimate with high accuracy the  winning-probabilities. Based on the results obtained, we conjecture that \emph{all} strategies, except for the strategy maximizing the winning probability of the classical (constrained) problem, converge to the random strategy under weak conditions on the number of players or empty boxes. We conclude by commenting on the possible applications of our results in understanding processes of information retrieval, such as ``memory'' in living organisms. 
 
\end{abstract}

\maketitle

\section{Introduction}

The aim of this paper is to study, by analytical and computational means, mathematical models of information retrieval.
Our starting point is a generalization of the problem originally proposed by G\'al and Miltersen \cite{GM} in the context of data structures:
\begin{quote}
Consider a team of $n$ prisoners $P_1,\dots,P_n$, $n$ keys labeled $1,\dots,n$, distributed randomly in $N \ge n$ boxes so that each box contains \emph{at most} one key. Each player $P_i$ is allowed to open $a\le N$ boxes to find the key $i$. The players cannot communicate after the game starts and the team wins if \emph{at least} $w \le n$ players find their key.
\end{quote}
We dub this problem as the \emph{unconstrained 100 prisoners problem} or the prisoners-search-game (PSG). The original motivation for the problem arose in computer science, concerning the trade-off between space (intended as storage space) and time (needed to perform a given task) for substring search algorithms. 
In colloquial terms, it asks what are the most efficient schemes to retrieve information encoded in data structures in which the information has been randomly stored.
While it was clearly important in that context that the information be retrieved completely, i.e. the team wins if all players find their keys, this puzzle has an obvious generalization whereby the team wins if a number $w\le n$ players find their keys in $a$ attempts. 
The general analysis of strategies for partial information retrieval after random and constrained storage, treated in this paper constitutes a necessary step towards a precise description of complex memory processes in living organism as well-defined PSGs. 

We consider a variety of strategies, each determines a family of probability distributions $P_S(a, w)$ as the maximum number of attempts $a$ is varied for a fixed strategy $S$. This family can be thought of as a function of $a$ and $w$ which we will henceforth call a \emph{P-function}: it gives the probability that following the strategy $S$, exactly $w$ players win within $a$ attempts.\par
 We first analyze three main classes of strategies in the case $N=n$, i.e., when there are no empty boxes: the random strategy, the key strategy (also called the \emph{pointer-following} strategy), and the box strategy in which players open the boxes in an arithmetic progression. 
We observe, both theoretically and experimentally, that for large $N$, the box strategy approximates the random strategy whose P-function is given by the binomial distribution in \eqref{eq:random_exact}. Furthermore, while the key-strategy remains \emph{on average} the best strategy confirming the results of the constrained classical problem, its minimum-winner P-function is actually smaller, in certain parts of the $a$-$w$ domain, than that of the random strategy!\par 
We also consider a variety of hybrid strategies for $n=N$ and $n<N$, obtained combining these three main strategies (box-,key-,random-) in various ways. As analytical methods become increasingly difficult in this setting, we turn to numerical experiments with Monte Carlo methods and quantify each strategy in absolute and relative terms, \eqref{efficiency} and \eqref{err} respectively. 
Using these tools, we compare the hybrid strategies we created with those of Avis-Devroye-Iwama (ADI) \cite{Avi} and Goyal-Saks (GS) \cite{GS}, both key-based algorithms. We experimentally verify that some of the former strategies are more \emph{efficient} (more winners with less attempts) than the latter.
Finally, based on our simulations, we formulate in Conjecture \ref{conj} the expectation that all properly bounded strategies, except the original key-strategy, will approximate the random strategy whenever $N$ grows large, or as $n/ N$ tends to zero. 
\subsection{Outline of the paper} The paper is organized as follows: in section \ref{sec:2}, we introduce the general set-up of the problem, comment on the constraints that can be imposed on the prisoners' choices, and present some important definitions; in section \ref{sec:3} we present and discuss the three most general classes of strategies logically allowed to solve the classical PSG and obtain, analytically whenever possible, the winning probabilities of the random, key, and box strategies. Generalized hybrid strategies, possible for $n=N$ and necessary when empty boxes are present, i.e. $n<N$, will be analyzed in section \ref{sec:4} along with implementations of the ADI \cite{Avi} and GS \cite{GS} algorithms. We then present a summary of all our results in section \ref{sec:5}, formulate Conjecture \ref{conj}, and finally discuss in section \ref{sec:6}.
possible applications to mnemonic processes of information-retrieval in living organism endowed with memory. 
We also include a set of appendices where we present some more details. 

\subsection{}

All the codes used in this work are freely available as a part of a Python package \texttt{prisoners-search-game}\cite{psg_package}.
\section{Set-up}
\label{sec:2}
\subsection{Definitions}
\label{subsec:2.1}

The problem we shall consider in this paper is a simple, yet broad generalization of the classical PSG  and can be stated as follows:
\begin{quote}
What is the optimal strategy for (at least) $w\le n$ prisoners to find their key by opening $a\le N$ boxes, knowing that the keys have been distributed uniformly randomly inside the boxes and each box can contain at most one key?
\end{quote}
In this paper, we assume there are at least as many boxes as players, $n\le N$. We will mostly consider the case $N=n$, leaving the variant $n<N$ for the second part of section \ref{sec:4}.  Note that if $w<n$, a distinction must be made between the probabilities for an \emph{exact} and a \emph{minimum} number of winners in $a$ attempts (P-functions), indicated by $P(a,w)$ and $P^{\rm min}(a,w)$ respectively.  It is easy to see also that $P^{\rm min}(a,N)=P(a,N)$, since $N$ is the upper bound on the number of winners. 

There are in principle infinitely many strategies to approach the problem for $N$ prisoners, each of which has $0<a \le N$ attempts  and whereas the group wins if (at least) $0<w\le N$ prisoner find their key. Let $b_{i,j}$ be the box opened by prisoner $i$ at attempt $j$ and $k_j$ the key found by the prisoner in box $j$. Any strategy must follow two algorithmic steps:
\begin{enumerate}
\item[S1.] Choose the first box to open: prisoner $P_i$ decides an offset ${D}\in \mathbb Z$ from the box $i$, and opens first the box $(i+{D})\bmod N$. There are two possible choices for the offset,
\begin{itemize}
\item ${D}=0$, each prisoner $P_i$ opens box $b_{i,1} = i$ first,
\item ${D} = D_i$, each prisoner $P_i$ opens box $b_{i,1}=(i+D_i) \bmod N$ and $D_i\neq0$ for at least one $i$. 
\end{itemize}
\item[S2.]Choose the remaining $j=2,\dots,a$ boxes, in search for the key. To do so we consider three possibilities, which can all be expressed in terms of a nonzero increment ${I}_{i,j}$, such that $b_{i,j+1}\equiv (b_{i,j}+{I}_{i,j}) \bmod N$ with: 
\begin{enumerate}
\item[Key strategy]: ${I}_{i,j}=(k_{j}-b_{i,j})$, the $(j+1)^{\rm th}$ box prisoner $i$ opens is decided by the number on the key found at the $j^{\rm th}$ box opened, i.e. the prisoner opens the box suggested by the number of the key found in the previous box.
\item[Box strategy]: ${I}_{i,j}=d$: the $(j+1)^{\rm th}$ box each prisoner opens is decided by constant shift $d$ of the box opened at the $j^{\rm th}$ attempt.\footnote{This assumption can be extended, so that the $(j+1)^{\rm th}$ box opened may depend on more than just the previous opened box, but an some/all previously opened boxes. As we shall comment below, relaxing this hypothesis will not modify the P-functions.} 
\item[Random strategy]: ${I}_{i,j}=d_{i,j}$: the $(j+1)^{\rm th}$ box each prisoner opens is decided by a non-constant shift $d_{i,j}$, which depends on the prisoner and the attempt.\footnote{There exist two middle cases, whereby the increment does (resp. does not) depend on the prisoner but it does not (resp. does) depend on the attempt. However, these are automatically accounted for by the most general case ${I}=d_{i,j}$.}
\end{enumerate}
\end{enumerate}
Here henceforth we will use the acronyms KS, BS and RS to indicate the above 3 strategies, respectively. We shall also use the superscript 0, as in KS$^0$ to indicate that the offset ${D}=0$ for the key strategy. 
Note that the case KS$^0$ represents the optimal solution to the classical 100 prisoner problem. 

Any strategy $S$, in principle, is a function of the players $P = \{P_1,\dots, P_n\}$, boxes $B=\{b_1,\dots,b_N\}$, and the keys $K=\{k_1, \dots , k_N \}$ found inside the boxes, with $B, K \subset \{1,2,\dots , N\}$. So generally speaking, we have a map \textcolor{red}{when $D= 0$}, 
\[
S: P\times B\times K \to B.
\]
The three possibilities for $S2$ discussed above, combined with a (specific) choice of $D$ for $S1$, give rise to three natural classes of strategies, which we analyze in the next section. We will be interested mainly in strategies such that no prisoner revisits a previously opened box, regardless of choices of $S1$ and $S2$. This reasonable assumption will often constrain the offset $D$ and/or increment ${I}$ to take on only specific values. It is then helpful to give the following:

\begin{defi}
Let $S$ be a strategy. We say:
\begin{enumerate}
\item
$S$ is \emph{properly bounded} if all $n$ prisoners will necessarily find their key for all $a\ge N$.
\item
$S$ is \emph{bounded} if all prisoners will necessarily find their key for all $a \ge M$, for some integer $M>N$.
\item
$S$ is \emph{unbounded} if not all prisoners necessarily find their key for any (finite) $a$.
\end{enumerate}

\end{defi}
We see that $S$ is not properly bounded if and only if players can open the same box more than once. Hence properly bounded strategies force the increment ${I}$ (or the offset ${D}$) to be such that no box is opened more than once. It follows by definition then that a properly bounded strategy $S$ has higher probability of success for an \emph{individual player} than any of its unbounded variants $S^\prime$, 
\[
P^i_S\ge P^i_{S^\prime},
\]
where $P^i$ specifies the probability of success of the player $P_i$.  On the other hand, if we are interested in the probability of success for a group of \emph{exactly} $w$ players this inequality does not hold between the P-function $P_S(a,w)$ and $P_{S^\prime}(a,w)$ in the whole $a$-$w$ plane. The reason is that both these P-functions reach maxima (minima) not only when the probability for $w$ players to win is maximized (minimized), but also when the probability for $N-w$ players to lose are maximized (minimized) (or equivalently, the probability for $N-w$ players to win are minimized (maximized)). However, by considering the P-functions for the group of \emph{at least} $w$ players to win, the remaining $N-w$ players are not required to lose. Consequently, we recover an inequality between the P-functions for a minimum number of winners:
\[
P^\text{min}_S \ge P^\text{min}_{S^\prime}.
\]
In appendix \ref{ref:appA} we present an unbounded variant of the random strategy analysed in section \ref{subsec:3.1} and show the correctedness of the above inequalities.

Aside from increasing the probability for individual players to win, there is another reason why the properly-bounded property for strategies is important, though the above distinction has been somehow hidden in previous approaches to this problem which fixed $w=N$. The reason is immediately apparent through the below lemma.  Note that the symmetry assumption turns out to be true for  both random and box-strategies, as we will show later in section \ref{sec:3}, equation \eqref{eq:symmetry} and appendix \ref{AppB}, Lemma \ref{eq:symm_box}.

\begin{lem}
If $P(a,w)=P(N-a,N-w)$, then the probability of success for \emph{at least} $w$ prisoners satisfies the identity
\[
P^{\rm min}(a,w)+P^{\rm min}(N-a,N-w+1)=1
\]
\end{lem}
\begin{proof}
Consider the left hand side of the above identity, and write it down explicitly in terms of the exact probabilities, hence
\begin{align*}
& \sum_{i=0}^{N-w} P(a,w+i)+\sum_{i=0}^{w-1} + P(N-a,N-i)
\\
&=\sum_{i=w}^{N} P(a,i)+\sum_{i=0}^{w-1}  P(a,i)
\\
&=\sum_{i=0}^N P(a,i)=1
\end{align*} 
where we have used the symmetry assumption to rewrite the second summation in the second line. The proof can also be followed in the opposite direction.
\end{proof}

\noindent In particular, under the properly-bounded constraint, if the exact probability function possesses the diagonal symmetry above: the probabilities $P^{\rm min}$ of half the $a$-$w$ plane determines the probability for the other half of the plane.

\subsection{Monte Carlo sampling}
Our aim in this paper is to explore the behaviour of $P(a,w)$ over a range of $a$ and $w$. We intend to study the problem for large $N$ and since analytic results are not available, we would like to compute the probabilities numerically by accounting for all possible permutations of the keys in the boxes. However, in this case it is difficult to obtain exact winning probability of a given strategy, since the total number of possible permutations of a set of $N$ numbers is $N!$.\par
To approach this problem rigorously, we will make use of Monte Carlo simulations to  describe a variety of different strategies.
The basic idea of Monte Carlo methods is to sample the underlying space of configurations (too large to be considered fully) and then work with these sampled points, or simply ``sample.'' Provided that the sample is drawn uniformly randomly, its properties must approximate the properties of the original space more and more as we increase the size of the sample (i.e. as we draw more and more permutations).\par 
To apply the Monte Carlo technique to our problem, for a fixed value of $N$, $n$, and $a$, we draw a fixed sample of permutations of $\{1,2,\cdots,N\}$ uniformly randomly from the set of all $N!$ permutations, and then simulate the game for each of those permutations with a given strategy. This allows us to estimate the P-function $P(a, w)$ to a great accuracy provided that an enough number of samples are drawn. At this point we assume the conventional $95\%$ confidence level on the probability $P$, which implies a \emph{z-score} $z_{95}\sim 1.96$. Furthermore we take $\sigma^2= p(1-p)$ to be the variance for the (possible) future samples of permutations, where $p$ is the percentage of permutations in a specific sample satisfying the condition ``exactly $w$ prisoners find their key within $a$ attempts''. We can then compute the margin of error $\mathcal{M}$ as:
\begin{equation}
\mathcal{M}=z_{95} \sqrt{\frac{p(1-p)}{s}}
\end{equation}
with $s$ sample size. It is easy to show that, under the above condition, by fixing $s\sim 10^4$, and because $0\le p \le 1$, the margin of error will be always $\mathcal{M} \le 1\%$. Hence, we will consider sample of size $s=10^4$ to be sure of the accuracy, within small margin of errors of our Monte Carlo simulations.
We remind that all the codes used in this paper are freely available as a part of a Python package \texttt{prisoners-search-game}\cite{psg_package}.

\section{Three main strategies}
\label{sec:3}
We now present the most general classes of properly bounded strategies discussed above. We start by considering the strategy based on random choices, for which the exact winning probabilities can be easily derived. This will be considered the benchmark we will compare all other results against. We then continue by  considering the key- and box-strategies.
\subsection{Random Strategy RS}
\label{subsec:3.1}
The random strategy is the benchmark of all properly bounded strategies.
For all $a$ attempts to be random, the offset $D$ must be random for every prisoner. The algorithmic steps are: 
\begin{enumerate}{\tt 
\item Player $B_i$ opens a random box from $1$ to $N$, say $j$.
\item If Box $j$ does not contain the key $i$, open another box randomly among the boxes not opened so far.
\item Repeat until the key $i$ is found OR stop after $a$ attempts.}
\end{enumerate}
The analytic formula for the probability of \emph{exactly} $w$ players finding their key after $a$ attempts each is simply given by
\begin{equation}
\label{eq:random_exact}
P_{\rm RS}(a,w)={N\choose w}\Big(\frac{a}{N}\Big)^w \, \Big(\frac{N-a}{N}\Big)^{N-w} 
\end{equation}
where ${N\choose w}$ is the binomial coefficient which computes the number of winning $w$-uples that can be formed from $N$ elements, the second factor gives the probability for $w$ prisoners to find their key after $a$ attempts, and the last factor gives the probability of $N-w$ prisoners to not find their key in $a$ attempts. 

We further note that the above probability \eqref{eq:random_exact} is symmetric under exchange $(a,w) \rightarrow (N-a,N-w)$, the property of reflection with respect to the center point $(\frac{N}{2},\frac{N}{2})$,
\begin{equation}
\label{eq:symmetry}
P_{\rm RS}(a,w)=P_{\rm RS}(N-a,N-w)\;,
\end{equation}
as can be shown using the symmetry of the binomial coefficient. Note that because of the above identity, the probability for half the $a$-$w$ plane determines the probability for the other half of the plane (see figures \ref{fig:random_main}).

\begin{prop}
\label{randprob}
The random strategy for $a$ attempts and $w$ winners has $P^{\rm min}$-function 
\begin{equation}
\label{eq:random_min}
P^{\rm min}_{\rm RS}(a,w)=\sum_{w^\prime=w}^{N} P(a,w^\prime)=\frac{1}{N^N}\sum_{w^\prime=w}^N {N\choose w^\prime}\,a^{w^\prime}\,(N-a)^{N-w^\prime}.
\end{equation}
\end{prop}
\noindent The proof of this follows immediately from \eqref{eq:random_exact}. In Figure \ref{fig:random_main} we show the heat-maps of the P-functions, i.e. the probability of winning, of the random strategy. The bifurcation that we observe can be seen as illustrating the cumulative distribution function for the binomial distribution,
\begin{equation}
P^{\rm min}_{\rm RS}(a,w+1) = 1 - P(X_{\rm RS}\le w)
\end{equation}
where the random variable $X_{\rm RS}$ is defined as
\begin{equation}
\label{XRS}
X_{\rm RS} = \sum_{i=1}^n \sum_{j=1}^a X_{ij},
\end{equation}
where each $X_{ij}$ is either 0 or 1 depending on whether or not the $P_i$ has found her key at the $j^{\rm th}$ attempt.
\begin{figure}
\begin{center}
\label{heatmap_exact_winners_randomclosed_N100_n100_natural}
\includegraphics[width=0.45\textwidth]{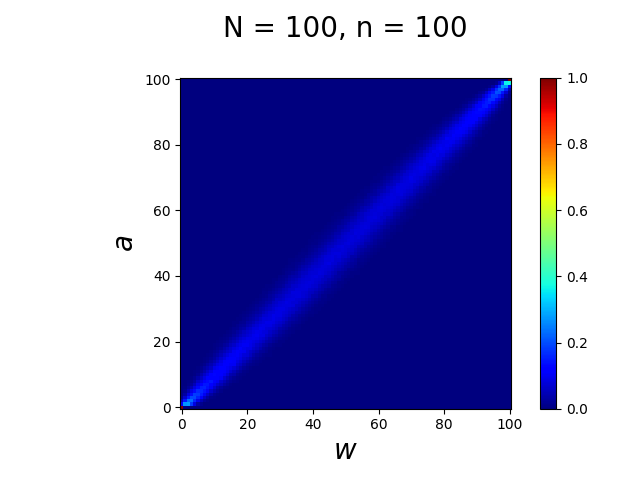} 
\includegraphics[width=0.45\textwidth]{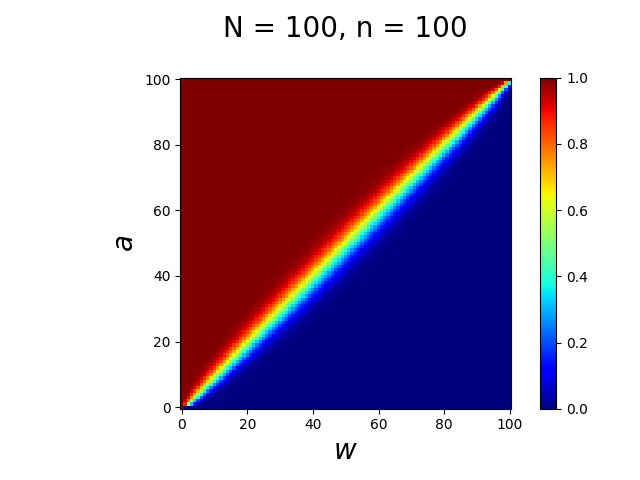} 
\end{center}
\caption{Random strategy P-functions. Left: exact winners heatmap. Right: minimum winners heatmap}\label{fig:random_main}
\end{figure} 
The exact winners P-function is peaked around the diagonal, with absolute maxima at  $(a,w)=(0,0),(100,100)$. This is to be expected because of the symmetry shown in \eqref{eq:symmetry}. More interesting is the minimum-winner P-function which is a very precise approximation for the Heaviside $\theta(x)$. Specifically, the graphs above show that
\begin{equation}
P_{\rm RS}^{\rm min}(a,w)\sim \theta(a-w).
\end{equation}
The random strategy can be considered as the benchmark for all the other strategies since it is reasonable to expect that any clever strategy should at least be as good as the random strategy. We will see how this reasonable assumption is encoded in our study later, when we will define a numerical estimator for the efficiency of a strategy.
\subsection{Key strategy KS$^0$}
The second strategy we analyze corresponds to the solution which maximizes the winning probability for $w=N$ players in the classical (constrained) problem \cite{CW}. We shall see that, when extended to the whole $a$-$w$ plane, its optimality persists but it is not evident as it was for the classical (constrained) problem.
In fact, the key strategy maintains a high success probability only in specific regions, a feature which to our knowledge has never been observed before. The key strategy KS$^0$ coincides with the following algorithmic steps:
\begin{enumerate}{\tt 
\item Player $B_i$ opens Box $i$.
\item If Box $i$ does not contain the key $i$, but key $j$, open Box $j$.
\item Repeat until the key $i$ is found OR stop after $a$ attempts.}
\end{enumerate}
In short, the strategy exploits the fact that the key-numbers and box-numbers create a structure of permutation cycles, which each player simply follows. Hence, for this strategy to be optimal, the offset is ${D}_i=0$ for all $i=1,\dots,N$. That is the only way a player is sure she will open only boxes-numbers which create, in combination with the numbers if the keys they contain, a cycle structure: a player, regardless from its identification number, will open only boxes, $a$ in total, in a permutation cycle surely containing her own key. Furthermore, a player $P^i$, with initial offset $D_i=0$, is sure to find her key in exactly $l$-attempts, if her key is in an $l$-cycle structure with the box-number.

The classical case then corresponds to $w=N$ and $a=N/2$, where the probability of success is given by
\begin{equation}
\label{eq:p_classical}
P_{\mathrm{KS}^0}(N/2,N)=1 - \sum_{k = N/2 + 1}^N\frac{1}{k} = 1 - \ln 2 + o(1)
\end{equation} 
which follows from counting the number of cycles $c$ of length $l_C\le a=N/2$, which produce exactly $l_c$ winners. Before proceeding, let us first recall the proof of this formula. Consider first the cycle decomposition of a given permutation of $N$ numbers. 
It can be written as a partition:
\[
N=\mathcal{N}_1+\mathcal{N}_2+\dots+\mathcal{N}_N    
\]
where $\mathcal{N}_k=k\cdot \alpha_k$ and $\alpha_k$ is the number of cycles of length $k$ present in the partition. In the following we will indicate a partition with $p$ and its cycle structure with $(\alpha_1,\alpha_2,\dots,\alpha_N)$ and unify the two notations by writing $p=(\alpha_1,\alpha_2,\dots,\alpha_N)$. Clearly, for $k>\lfloor N/2\rfloor$, the largest integer smaller than or equal to $N/2$, we have $\alpha_k=0,1$. Similarly, any permutation will probably have some $\mathcal{N}_k=0$, because $\alpha_k=0$.

Now, it is relatively easy to compute the number of permutations which contain at least (and obviously at most) a cycle of length $l>\lfloor N/2\rfloor$ can be calculated as follows:
first consider all possible ways to extract $l$ elements from $N$. For each of these $l$-elements, there exists $(l-1)!$ inequivalent permutations which combine into an $l$-cycle (easy proof by induction). Finally, we can consider all the permutations of the remaining $(N-l)$ elements. The final formula hence reads:
\begin{equation}
\label{eq:class_count}
\Omega_l^N={N\choose l} (l-1)! (N-l)!=\frac{N!}{l}.
\end{equation}
Dividing this formula by $N!$ we obtain the probability that a permutation contains an $l$-cycle, $l>N/2$, from which the formula \eqref{eq:p_classical} is derived.

Next, we are interested in extending this formula to the cases $w<N$ winners and $a$ unconstrained. The counting is not easy anymore, as cycles of length $l$ not greater than $\lfloor N/2\rfloor$ may appear with multiplicity $\alpha_l>1$. There exists however a formula which counts the multiplicity of a certain partition of cycles of $N$, $p=(\alpha_1,\dots,\alpha_N)$ \footnote{An extension of this formula, presented in \eqref{eq:recursive_Omega}, allows for an explicit expression for $\Omega^N_{\alpha_l=k}$ in terms of more easily computed numbers $\Omega^{k_1\cdot l}_{\alpha_l=k_1}$ and $\Omega^{N-\alpha_1\cdot l}_{\alpha_l=k-k_1}$. For the scope of this section, this formula will not be necessary.}:
\begin{equation}
\label{eq:perm_cycle_count}
\Omega^N_{(\alpha_1,\dots,\alpha_N)}=\Omega^N_p=\frac{N!}{{1^{\alpha_1}\dots N^{\alpha_N}\,\alpha_1!\dots \alpha_N!}}\;.
\end{equation}
To give a simple example, assume $N$ even, and we want to find how many permutations of $N$ are such that their partition of cycles is $p=(0,\dots, \alpha_{N/2}=2,0,\dots,0)$. There are $\Omega^N_p=\frac{N!}{(N/2)^2 \cdot 2!}$ such permutations. For $N=4$, the result is simply $\Omega^4_{(0,2,0,0)}=3$, since there are 3 permutations of the numbers $(1,2,3,4)$ which contain two 2-cycles: $(2,1,4,3)$, $(3,4,1,2)$ and $(4,3,2,1)$.\par
We also note that, given a partition $p$ of $N$ as in the above, into cycles of length $i$, each cycle of length $i \le a$ will produce ${\mathcal N}_i$ winners, while each cycle of length $j >a$ will produce $j$ losers. Hence the total number of winners for a given partition $p$ of $N$ is given by:
\begin{equation}
\label{partsum}
w_p(a)=\sum_{{\mathcal N}_i\le a} {\mathcal N}_i\,.
\end{equation}
From the preceding discussion, we obtain the general formulas below.
\begin{prop}
\label{keyprob}
The key strategy for $a$ attempts and \emph{exactly} $w$ winners has P-function 
\begin{equation}
P(a,w)= \frac{1}{N!}\sum_{p} \Omega^{N}_{p}\cdot\delta\big(w,w_p(a)\big)\;, 
\end{equation}
where $\delta$ is the Kronecker delta, equal to $1$ only when $w=w_p(a)$.
Similarly, the P-function for the key-strategy with $a$ attempt and \emph{at least} $w$ winner reads:
\begin{equation}
P(a,w)= \frac{1}{N!}\sum_{p} \Omega^{N}_{p}\cdot\theta\big(w-w_p(a)\big)\;, 
\end{equation} and 0 otherwise.
\end{prop}
In the above equations the sum is extended over all partitions $p$ for which the number $w_p(a)$ in \eqref{partsum} is equal-to or no-less-than $w$.
In figure \ref{KSplots} we show the P-functions of the key strategy, which possess unique features: for small values of both $a$ and $w$, both P-functions resemble the random P-functions (with a larger spread) while the minimum-winner P-function is symmetric under exchange $a\leftrightarrow w$ around the second diagonal.\\
\begin{figure}[t]
\begin{center}
\includegraphics[width=0.45\textwidth]{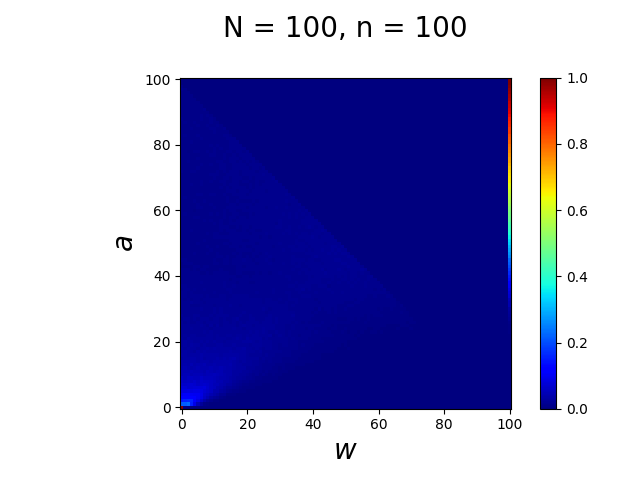} 
\includegraphics[width=0.45\textwidth]{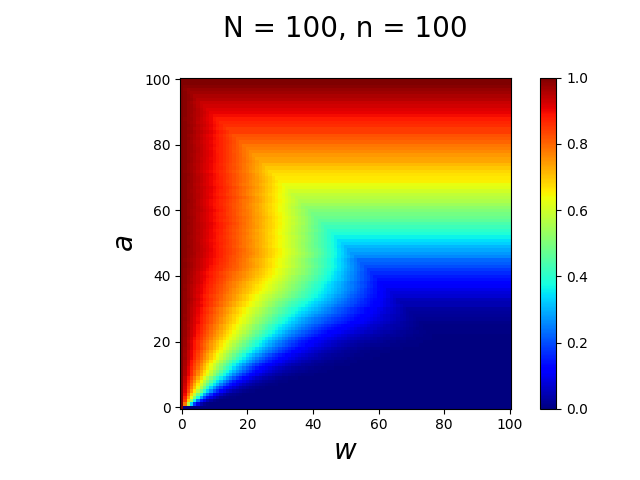} 
\end{center}
\caption{\label{KSplots} Key strategy P-functions. Left: exact-winner heatmap. Right: minimum winner heatmap.}
\label{fig:key-main}
\end{figure} 

\subsection{Box strategy BS}
Finally, we discuss the simplest strategy to implement, the box strategy. As we will observe heuristically and experimentally, the choice of non-zero offset(s) is immaterial to the box strategy P-functions, but for clarity of exposition, let us consider first the case $D_i=0$. The algorithmic steps are:  
\begin{enumerate}{\tt 
\item Player $P_i$ opens Box $i$.
\item If Box $i$ does not contain the key $i$, open Box $(i+{I})$ mod $N$
\item Repeat until the key $i$ is found OR stop after $a$ attempts.
\item[(${C}$)] The increment ${I}$ is coprime to $N$.}
\end{enumerate}
The above condition (${C}$) can be explained as follows: if $\gcd({I},N) = d>1$, then each player will re-open the same box previously opened in attempt $N/d$, which could be considered a ``period''. This of course renders the strategy unbounded, significantly decreasing the probability of success. Hence we will not consider this unbounded variant, but instead always impose condition (${C}$) on the box strategies.
\begin{rem}
Although box strategies clearly aim to put an order in the selection of the boxes, they are by all means sub-cases of the random strategy. Unfortunately, neither strategy is equipped to ``unravel'' the random distribution of keys, i.e. to find an order in the disordered key positioning. The randomness of the keys' positioning is then the leading factor to be considered when calculating probabilities with box and random strategies. Whether the box selection is ordered or not, it is of little importance since the keys have been randomly distributed. Indeed, as we shall soon show by means of analytical and statistical approaches, the box strategy converges to the random strategy for $N$ sufficiently large (see Figure \ref{fig:box_str})
\end{rem}
\begin{figure}[h]
\begin{center}
\includegraphics[width=0.45\textwidth]{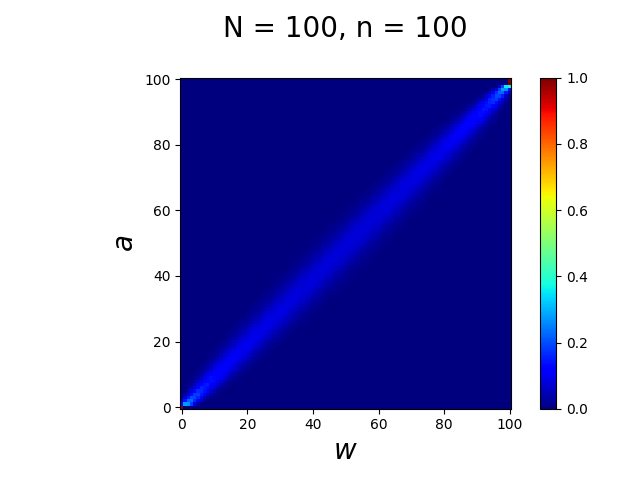} 
\includegraphics[width=0.45\textwidth]{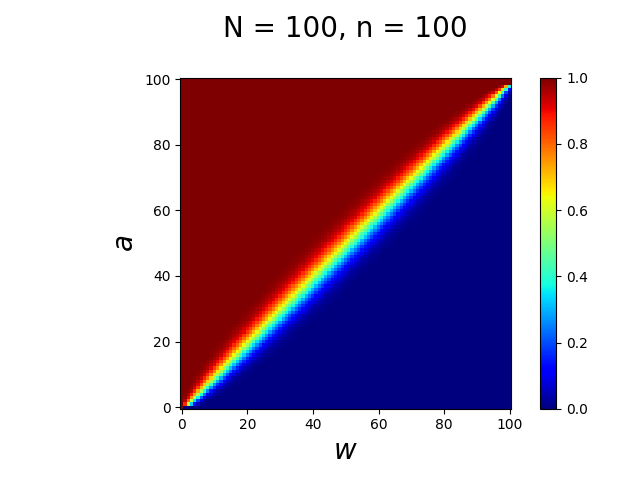} 
\end{center}
\caption{Box strategy P-functions with increment 1, offset 0. Left: exact winners heatmap. Right:minimum winners heatmap.}
\label{fig:box_str}
\end{figure} 
To compute the probability of winning with this strategy, we need to enumerate  permutations with restricted positions. For clarity of exposition, let us first take the offset ${D}_i=0$ for all $i$, and generic increment $I$ coprime to $N$. This means that, for player $i$ to win, her key $i$ needs to be in a box between $i,i+I \dots i+(a-1)I$. This problem is complementary, in a way succinctly expressed in the proof of Lemma \ref{lem:box},  to the famous ``probl\`{e}me des rencontres'' \cite{2disc}, ``probl\`{e}me des m\'{e}nages'' \cite{3disc}, and generalizations thereof, more generally known as the counting of restricted permutations, often expressed in the language of chessboards and rooks positioning \cite{riordan}.
 These problems have a long history and are notoriously hard to solve: to the best of our knowledge, a closed formula for the number of $p$-discordant permutations, i.e. permutations for which \emph{no key} $i$ is in its box $i$, or in any successive boxes until $i+p-1$, exists only up to $p=5$ \cite{5disc}.\par
Here we will present the results for the first two cases $a=1,2$, which are related to 1- and 2-discordant permutations  respectively.
\begin{prop} 
\label{boxprob}
For any $1\le w \le N$, we have
\[
P_\mathrm{BS} (1,w) = \frac{1}{w!}\sum_{k=0}^{N-w}\frac{(-1)^k}{k!},
\]
and
\[
P_\mathrm{BS} (2,w)=\frac{c_w}{N!}=\frac{1}{N!}\sum_{k=w}^N (-1)^{k-w} \,\frac{2N}{2N-k}{2N-k \choose k}(N-k)! {k \choose w}.
\]
\end{prop}
For a proof of the above two formulas we refer the reader to Appendix \ref{proof:prop3.4}.

\begin{cor}
$P_\mathrm{BS} (1,w)$ converges to $P_{\rm RS}(1,w)$ for $N$ large enough.
\end{cor}

\begin{proof}
This follows from the relation $P_\mathrm{BS} (1,w) = {D_{N,w}}/{N!}$ in \eqref{eq:Dnw} and \eqref{eq:Dn0}, and fact that 
\[
\lim_{N\to\infty}\frac{D_{N,w}}{N!} = \frac{e^{-1}}{w!},
\]
so the limiting distribution is equal to the Poisson distribution with expected value 1, which is the same limit as the binomial distribution.
\end{proof}
\begin{rem}As we have mentioned, analytic formulas exists for 3-,4- and 5-discordant permutations, namely for $a=3,4,5$ attempts (we refer to \cite{3disc,4disc,5disc} for details). In Figure \ref{fig:box_str} we show the exact- and minimum-winner P-functions for the box strategy with increment 1 and zero offset. As we observe, the P-function appears to be identical (within small margin of errors) to the random-strategy P-function, as we expected and argued above.
\end{rem}

Since the offset has been set to zero for simplicity, $P_i$ can only find their key within boxes $(i, i+1,\dots i+a-1)\bmod N$. Now consider the simple case $a=1$. It is clear that as long as prisoners are assigned a box number in a one-to-one correspondence, it  does not matter which box number that is, but the above arguments can be applied after a simple re-ordering of box (or player) numbers, e.g., $P_i$ can be assigned box $i+{D}$ as her own fixed point. For the case $a=2$, it does not matter whether $P_i$ opens box $(i,i+1) \bmod N$ or $(i+{D},i+{D}+{I})\bmod N$.

 Similarly, it does not depend on the fact that some prisoners may decide to change their offset and/or increment. This is the case, for instance, where the box $b_{i,j+1}=f(\{b_{i,j}\})$, where $f$ is a bijective function on $[1,\dots, N]$. This strategy would constitute an intermediate case between box strategy and random strategy, but since their P-functions are (almost) identical, a strategy that interpolates between the two will necessarily have (almost) identical P-function. It is natural then to ask whether all properly bounded strategy $S$, that is independent of the key-numbers $K$, will asymptotically approach the random strategy RS. Our many simulations have confirmed this to be true. We can then present the most general, naturally properly bounded strategy, not based on the key numbers, BS:
\begin{enumerate}{\tt 
\item Player $P_i$ opens Box $(i+{D}_i) \bmod N$.
\item If Box $(i+{D}_i)\bmod N$ does not contain the key $i$, open Box $(i+{D}_i+{I}_i)\bmod N$, with ${I}_i$ such that no box is opened more than once.
\item Repeat until the key $i$ is found OR stop after $a$ attempts.
\item[(${C}$)] The increment ${I}_{i,j}$ should be such that no boxes are opened more than once.}
\end{enumerate}
We will formulate this `convergence' to the random strategy in more general terms in Conjecture \ref{conj} below.

Finally, even though analytic results are not generally known, approximations for the enumerators of discordant permutations have been studied.  For instance, \cite{riordan} shows that if $(a-1)< N^{1/3}$, the normalized probability $P_{\rm BS}^N(a,w)$ can be expanded in inverse powers of $n$ (or $(n)_r=n\cdot (n-1)\cdots (n-r+1)$), with leading order term corresponding to a Poisson distribution. This suggests that, for $N$ large enough, the $P_{\rm BS}$ will approximate $P_{\rm RS}$, as it was hinted at earlier.

Let us give a heuristic argument to justify this. As we have seen, the random positioning of the key can be ``compensated'' by a smart choice of the strategy exploiting the cycle-decomposition of any random permutation of keys. On the other hand, without exploiting this cycle decomposition, it is hard to devise a way by which to find the keys in a specific order with any certainty. Hence, lacking a smart strategy to unravel the random distribution of keys, \emph{any} other strategy will necessarily result in an approximate random probability of winning: even if the choice of boxes follows a certain order, as in the box strategy described above, the probability of winning will just reflect the randomness of the key distribution. If the choice of boxes is also random, then the probability will remain random, since a random shuffling of randomly distributed elements will still give a random distribution. This approximation will be more and more valid as the number of randomly shuffled elements increases.  For small values of $N$, there may be finite differences, but they become negligible as $N$ grows large. We shall describe this more concretely in \eqref{sec:conj} below.

\section{\label{hybrid}Hybrid strategies}
\label{sec:4}
\subsection{Escape routes}
In the previous section we analyzed the three general classes of properly bounded strategies to approach the PSG. Crucially, for the box-strategy and the random strategy to be properly bounded, a constraint ${C}$ was required. In the case of the key strategy, it is the condition ${D}=0$ that makes the strategy properly bounded. To understand why, imagine player $P_i$ first opens the box $j\neq i$. Now, $P_i$ can either be inside her cycle or inside a different cycle. In the latter case, $P_i$ may enter a cycle of length smaller than $a$ and hence be forced to re-open certain boxes, never to find hers.

If for some reason the constraints on the increment (BS) or condition on the offset (KS) are not satisfied, the strategies become unbounded. Nevertheless, it is possible to recover proper boundedness by adding an extra algorithmic step, an escape route E, to all strategies:
\begin{enumerate}
\item[E]: If the next box to be opened has already been opened, choose another one to open, among the un-opened boxes.
\end{enumerate}
This choice can again be made in one of three ways:
\begin{itemize}
\item choose the next box as indicated by the key number, $E=K$
\item choose the next box randomly among the boxes not opened so far, $E=R$
\item choose the next box sequentially, $E=B$
\end{itemize}
Note however that the first choice above does not necessarily, and on its own, imply boundedness. So we will neglect for the time being the key strategy as an escape route, though we will later present a (bounded) strategy that contains the KS as an escape.
Another important feature, strengthened by the results of the previous section, is that the difference between a sequential or random choice of box to open is, within small statistical errors, inessential: as long as the key are distributed randomly, the P-function for the RS and BS are almost identical. This seems to suggests that the difference between random and sequential choice of the next box to open is small enough to be neglected. This fact seems to be confirmed by the simulations (we encourage the curious reader to check this statement making use of \cite{psg_package}), though as we shall see, our efficiency index will pick up on small differences between strategies with a box or a random strategy as an escape.

For the time being we can, without loss of generality, consider the P-functions for the hybrid strategies KS and BS complemented by a random choice of the next box to be opened as an escape route. Note that, given our definition at the opening of section \ref{sec:3}, the random strategy will never require an escape route \footnote{In fact, the hybrid strategy random with a random escape route corresponds by definition to the bounded random strategy, whose P-function has been shown in figure \ref{fig:random_main}}, as the constraint assures proper boundedness. Hence, we can still consider its P-function as the benchmark of all hybrid strategies with $n=N$ .

Further extensions of the game are possible. Of particular interest in practical applications is the case $n<N$, for which some boxes will not contain a key, they are empty. In this case, an escape route will be needed while using the key-strategy if the last opened box is empty. We will also show an example of a bounded hybrid strategy \cite{GS}, the Goyal-Saks algorithm, which uses $E=K$ escape route for the BS with $I=1$ and a surplus. Finally, we will present simulations of a properly bounded strategy which uses a notion of fictitious keys as an escape route for the KS \cite{Avi}.
\subsection{$n=N$ hybrid strategies}
The first hybrid strategy we consider is the unbounded key strategy, ${D}\neq 0$, to which we add a random choice as an escape route whenever the player is about to open an already-opened box. The algorithmic steps are:
\begin{enumerate}{\tt 
\item Player $P_i$ opens any Box $j\neq i$.
\item If Box $j$ does not contain the key $i$, but key $k$, open Box $k$.
\item Repeat until the key $i$ is found OR stop after $a$ attempts.
\item[(E)] If the Player is forced to open a box she already opened, pick the next box randomly among the unopened boxes}
\end{enumerate}
We present in figure \ref{fig:key+rand} results of the Monte Carlo simulation for this hybrid strategy.

As we see immediately, the P-function for the hybrid of a key strategy with a random (or sequential) escape route approximates with high precision the random strategy P-function in figure \ref{fig:random_main}. 
The next hybrid strategy we consider is the unbounded box strategy, i.e. the increment ${I}$ is not coprime to $N$, with $d={\rm gcd}(I,N)$ complemented by a random-choice escape strategy. We call $\tau=N/d$ steps the \emph{period}, the number of attempts after which all prisoners will be forced to open again an already-opened box. The algorithmic steps describing this BS hybrid are:
\begin{enumerate}{\tt 
\item Player $P_i$ opens any Box $j$.
\item If Box $j$ does not contain the key $i$, open box $j+I \bmod \; N$, $I$ not coprime to $N$
\item Repeat until the key $i$ is found OR stop after $a$ attempts.
\item[(E)] If the Player is forced to open a box she already opened, pick another box randomly among the unopened boxes}
\end{enumerate}
The P-function is plotted in figure \ref{fig:box+rand}. It again corresponds, with high accuracy, to the random strategy P-function.
\subsection{$n<N$ hybrid strategies}
The case of prisoner search games with empty boxes is one of the most explored variants in the literature. It clearly necessitates of an escape route to balance the presence of empty boxes and still enforce proper boundedness. We first consider a benchmark for these strategies, the properly bounded random-strategy whose analytic P-function, shown in figure \ref{fig:random_n}, reads:
\begin{equation}
\label{eq:random_str_n}
P(a,w)= {n \choose w}\Big(\frac{a}{N}\Big)^w\Big(\frac{N-a}{N}\Big)^{n-w}\;.
\end{equation} 
In parallel to the previous subsection, we now define two types of hybrid strategies based on the unbounded key- and box-strategies.\par
 First, we consider the hybrid strategy obtained from the unbounded box strategy complemented by a random escape route (figure \ref{fig:box+random}). As one would expect simply by comparing the P-functions for random and box strategies, the exact- (and hence the minimum-)winners P-function for this hybrid will still be peaked only around the diagonal of equation: $a=\frac{N}{n}w$ (see figure \ref{fig:box+random} for the case $n=50$).

The second case of hybrid strategy we consider is obtained from a key strategy with $D=0$, KS$^0$, complemented by a random or sequential escape route. In this case, one can intuitively expect (the box-selection imposed by) this hybrid strategy to ``collapse'' to the random-strategy if $n\ll N$: many boxes are empty, forcing prisoners to opt for a random selection of the next box to open. This argument implies that for $N-n$ large enough, the hybrid of an unbounded key strategy will reduce to a random strategy. 
In figure \ref{fig:2hybr} we show that case $n=N-n=50$ and note some very small differences, not connected to statistical errors, for small values of $a$ and $w$.\footnote{The same hybrid strategy, based on the key strategy with random or sequential escape route, but non-zero initial offset $D$ gives instead efficiency slightly below $1$.}

Interestingly, even for $N-n=1$ (only one empty box), the P-functions for this hybrid strategy already possess the characteristic profile (along the diagonal) of the random-strategy P-functions, see figure \ref{fig:key+random_offset0} and \ref{fig:key+box_offset0}. 
\subsection{The Avis-Devroye-Iwama (ADI) strategy}
\label{ADI}
So far we have presented strategies easily obtained as combinations of the three main strategies, random, key and box. There are of course examples of hybrid strategies whose escape does not belong to any of the above three categories. Examples of such strategies were presented in \cite{Avi}, where it was initially assumed $n<N$ and that all prisoners are surely aware of the value of $n$ and $N$, and hence $N-n$ \footnote{Note that, while it was always in principle possible for the prisoners to know their total number $n$ and the boxes number $N$, all the strategies analyzed so far were constructed regardless of this information, i.e. all prisoners would be given a unique list of boxes to open, even without knowledge of the exact number of empty boxes, $N-n$}. We will consider here only the first example, named PF-1, since the second example, PF-2$(t)$, reduces to PF-1 for specific values of the parameter $t$ and has been shown to produce lower winning probability in general. The strategy PF-1 can be described by the following algorithmic steps:
\begin{enumerate}{\tt 
\item Player $P_i$ opens Box $i$ ($D=0$) and sets an index $j_i=0$, which counts the number of empty boxes that are opened during the search
\item If Box $i$ contains key $j\neq i$, open box $j$; if box $i$ is empty, increment $j_i$ by $1$, and go to open box $n+j_i$
\item Repeat step 2 until the key $i$ is found OR stop after $a$ attempts.}
\end{enumerate}

It is immediately clear why the information about the exact value of $N$ and $n$ is required, since otherwise the prisoner would not know which box to open after having opened an empty box. The index $j_i=1,\dots,N-n$ counts, in the same order they have been opened by each prisoner, the empty boxes. Since this opening order may vary from prisoner to prisoner, it is important to realize that the cycle structure of the keys, the union of the `real' keys numbered from $1$ to $n$, and the `fictitious' keys which force the players to open the boxes from $n+1$ to $N$ (for which there is no actual key present), is not unique in this case, but it definitely varies from prisoner to prisoner (We refer the reader to the explicit example given in \cite{Avi}). This of course does not happen when $N = n + 1$, only one empty box, since in that case all prisoners will agree on its box number and the P-functions will look exactly as Figure \ref{KSplots}. Hence, for this strategy alone, we will not show the case $n = 99$, but $n = 98$ instead. In figure \ref{fig:adi98} and \ref{fig:adi} we show the P-function for the strategy PF-1, for $n=98$ and $n=95$ respectively. For $n$ small enough, the P-function will approximate once again to the P-function for the random strategy, shown in figure \ref{fig:random_n}. 

\subsection{The Goyal-Saks (GS) strategy}
\label{GS}
Finally, we simulate the P-function for the Goyal-Saks algorithm in \cite[Theorem 1]{GS}, in the general case $w<n$. This hybrid strategy is bounded, but not properly bounded. The strategy of Goyal and Saks is as follows. Let $d=\lfloor{N/n}\rfloor$ and let us denote by $[s,t]$ the set of integers $\{s, s+1,\dots, t-1,t\}$ if $s\le t$ and $\{s,\dots, N\}\cup\{1,\dots t\}$ if $s>t$.  Define \texttt{occupied}$[s,t]$ to be the number of boxes that contain a key (hence are not empty) in $[s,t]$, and define 
\[
\mathrm{surplus}[s,t]= \mathrm{occupied}[s,t] - \frac{|s-t|}{d}.
\]
In other words, the surplus function measures how much the number of non-empty boxes in the interval $\{s,t\}$ differ from the average number of non-empty boxes in the same interval. For each $i$, we let $m(i)$ be the smallest integer such that surplus$[i,m(i)]$ is non-negative. Finally, we partition the boxes into bins $B_1,\dots, B_{n}$, where $B_i$ contains boxes $[d(i-1)+1,d(i-1)+d]$ for $i\in [1,n-1]$, and $B_n$ contains boxes $[d(n-1)+1, N]$.\footnote{Note that this is only one possible arrangement of boxes into bins when $N/n$ is not an integer.} The algorithm is then the following.
\begin{enumerate}{\tt
\item Player $P_i$ starts at the first box of the bin $B_i$.
\item Check boxes sequentially, keeping track of the surplus until surplus is non-negative, hence $m(i)$.
\item If box $m(i)$ contains the key, done.
\item If not, then box $m(i)$ will contain key $j$, go to bin $B_j$. Reset surplus and repeat.}
\end{enumerate}
This means that the player will follow a BS until the surplus is negative, then use the key strategy as an escape when the surplus becomes non-negative, then continue with the sequential selection of boxes, and so on. It is clearly a hybrid strategy, but not a properly bounded one.

Then the main result of \cite{GS} is that their strategy, GS, for $w=n$ and $a=N/k$ ($k$ also depends on $n$) has success probability at least 
\[
2^{-9\sqrt{kn}\log_2n -k}\;.
\]
 It is known that this lower bound can be improved to $2^{-3\sqrt{kn}\log_2n -2k\log_2 e/3},$ but it is ineffective in the following sense. The classical case corresponds to $N=n=100$ and $k=2$, and substituting into the latter expression we have
\[
2^{-3\sqrt{200}\log_2100 -4\log_2 e/3} \sim 2\times 10^{-85}\;,
\]
which is much smaller than the exact probability $0.311$ calculated from the classical solution. Indeed, it remains an open problem to determine whether the probability $P(N/2,n)$ with varying $n$ tends to zero. To answer this question, we simulated the GS strategy for different values of $ N/2\le n < N$ confirming that indeed $P(N/2,n)$ is vanishingly small. Given our limited samples of all permutations, it would be clearly impossible to verify with precision the bound (to obtain a probability of $10^{-85}$ we would have to consider at least $10^{85}$ permutations and  verify that in at most one of those can $n$ players win with $N/2$ attempts). \par
In figure \ref{fig:Goyal-Saks} we show the case $N=100$ and $n=99$: it is easy to see that the strategy is not properly bounded, but only bounded, since for $a=N$, the P-function $P(N,w)$ never reaches 1 (see exact winners histogram in figure \ref{fig:Goyal-Saks}).

\section{Convergence to the random strategy}
\label{sec:5}
So far, we have analyzed and presented the P-functions for a variety of old and new strategies for the resolution of the generalized 100 prisoner problem, where not only the number of attempts $a$, but also the minimum number of winners $w$ can be varied. Although the numerical plots illustrate clearly the differences/similarities between certain strategies, in the section below we aim to quantify the absolute \emph{efficiency} of a strategy, as well as the \emph{error} between two strategies, by means of two estimators of our own making.

\subsection{Estimators of a strategy}
The first estimator we introduce here, called the efficiency $\eta$ of a strategy, measures the performance of a given strategy in terms of producing the exact number of winners within the right number of attempts. In fact, it is logical to assume that a strategy is more efficient than another if it produces more (equal) winners with the same (less) number of maximum attempts. 
 In other words, a more efficient strategy is one for which $P(a, w)$ tends to take high values for low values of $a$ and high values of $w$. If we multiply each value $P(a, w)$ by the function $(w/a)^{\beta}$ with $\beta>0$, it is easy to see that for an efficient strategy, the product takes high values in the correct region (low $a$, high $w$) of the $a$-$w$ plane whereas for a less efficient strategy this will not be the case. Hence, if we sum this product over all the combinations of $a$ and $w$, it should tell us which strategy is efficient on average. Thus, we have
\begin{equation}
\label{efficiency}
\eta=\frac{1}{C}\sum_{a,w} P(a,w)\cdot \Big(\frac{w}{a}\Big)^{\beta},
\end{equation}
where $C$ is a normalization constant which is fixed below to correspond to the value of $\eta$ for the random strategy (normalization constant), \eqref{eq:random_exact} for $n=N$ or \eqref{eq:random_str_n} for $n<N$. 
It should also be clear that small values of $\beta$ may not be able to differentiate strategies well enough since $(w/a)^{\beta}$ would take similar values all over the plane. We find that $\beta = 2$ leads to a good resolution of the strategies, and henceforth we will fix this value for $\beta$.\\
\begin{table}[h]
    \begin{minipage}[l]{5cm}
        \centering
        $n=100$
        \vspace{0.5em}

        \begin{tabular}{|c|c|c|}
        \hline
            {\bf Strategy} & {\bf Escape} & $\mathbf{\eta}$ \\
        \hline
            KS$^0$ & - & 1.35 \\
        \hline
            KS & RS & 1.00 \\
        \hline
            KS & BS & 1.00 \\
        \hline
            BS$1$ & - & 1.00 \\
        \hline
            BS$5$ & RS & 1.00 \\
        \hline
            BS$5$ & BS & 1.00 \\
        \hline
            Goyal-Saks & - & 1.35 \\
        \hline 
            ADI & - & 1.35 \\
        \hline
            RS & - & 1 \\
        \hline
        \end{tabular}

    \end{minipage}

    \begin{minipage}[r]{5cm}
        \centering

        \vspace{0.5em}
        $n=99$
        \vspace{0.5em}

        \begin{tabular}{|c|c|c|}
        \hline
            {\bf Strategy} & {\bf Escape} & $\mathbf{\eta}$ \\
        \hline
            KS$^0$ & RS & 1.21 \\
        \hline
            KS$^0$ & BS & 1.26 \\
        \hline
            KS & RS & 1.00 \\
        \hline
            KS & BS & 1.00 \\
        \hline
            BS$1$ & - & 1.00 \\
        \hline
            BS$5$ & RS & 1.00 \\
        \hline
            BS$5$ & BS & 1.00 \\
        \hline
            Goyal-Saks & - & 1.12 \\
        \hline 
            ADI$^{\ *}$ & - & 1.30 \\
        \hline
            RS & - & 1 \\
        \hline
        \end{tabular}
    \end{minipage}
    \begin{minipage}[r]{5cm}
        \centering

        \vspace{0.5em}
        $n=50$
        \vspace{0.5em}

        \begin{tabular}{|c|c|c|}
        \hline
            {\bf Strategy} & {\bf Escape} & $\mathbf{\eta}$\\
        \hline
            KS$^0$ & RS & 1.01 \\
        \hline
            KS$^0$ & BS & 1.04 \\
        \hline
            KS & RS & 1.00 \\
        \hline
            KS & BS & 0.99 \\
        \hline
            BS$1$ & - & 0.99 \\
        \hline
            BS$5$ & RS & 1.00 \\
        \hline
            BS$5$ & BS & 1.00 \\
        \hline
            Goyal-Saks & - & 0.99 \\
        \hline 
            ADI & - & 1.00 \\
        \hline
            RS & - & 1 \\
        \hline
        \end{tabular}
    \end{minipage}
    \vspace{1em}
    \caption{\label{eff_table} Efficiency of various strategies described in the paper for three different values of $n$. Note that, since the values are obtained from P-functions simulations using Monte Carlo sampling, they are not exact. Above, we show only the statistically meaningful decimals for the estimated efficiency. ($^{*}$Here the $\eta$ value is written for $n=98$ since ADI for $n=99$ is identical to $n=100$)}
\end{table} 
In Tables \ref{eff_table}, we indicate zero offset $D=0$ by adding a 0 index to the acronym for the strategy, e.g. KS$^0$. Also, since the box strategy requires the specification of a (constant) increment $I$, we will indicate it as BS$I$ in the tables.

We have quantified how `favorable' (more winners with less attempts) each of the strategies are, using the efficiency index $\eta$ \eqref{efficiency} which is always normalized by a RS, \eqref{eq:random_exact} for $n=N$ or \eqref{eq:random_str_n} for $n<N$. 
 From the table \ref{eff_table}, it is clear that the box strategy with an arbitrary increment is only as efficient as the random strategy. The same is true for the key strategy when players do not start by opening their own box. It is also evident from the table that as the number of empty boxes increases, all strategies P-functions converge to the random strategy P-functions \eqref{eq:random_str_n}. \par
\par
The second index $\epsilon_{ij}$ directly compares how much two P-functions differ from each other by taking the sum of the absolute differences between their values for each combination of $a$ and $w$, and then averaging these differences over all possible values of $a$ and $w$: 
\begin{equation}
\label{err}
\epsilon_{ij} =\frac{1}{(n+1)N}\sum\limits_{a, w} |P_i(a,w)-P_j(a,w)|\;,
\end{equation}
where the normalization factor is easily obtained since $w$ takes $n+1$ values $0, 1, 2, \cdots, n$ and $a$ takes $N$ values $1,2,\dots, N$. We note that the error $\epsilon_{ij}$ is related to the usual variational distance, given by
\begin{equation}
\delta(P_i,P_j) = \frac{1}{2}\sum_{\omega\in\Omega} |P_i(\omega) - P_j(\omega)|,
\end{equation}
where $\Omega$ denotes the sample space, and $P_i,P_j$ are probability measures, by the formula
\begin{align}
\epsilon_{ij} =\frac{1}{(n+1)N}\sum\limits_{a, w} |P_i(a,w)-P_j(a,w)| =  \frac{2}{(n+1)N} \sum_{k=1}^a \delta(P_i(k,\cdot), P_j(k,\cdot)).
\end{align}
In other words, the error $\epsilon_{ij}$ can be viewed as an average distance between two families of probability measures determined by two strategies. With this in view, we observe that other measurements of error can be defined through different choices of statistical distance $\delta$.
In Fig.~\ref{fig:errors} in appendix, we show $\epsilon_{ij}$ heatmap for different combinations of strategies presented in this paper. It is easy to see from these how strategies with similar (if not identical) efficiency also have small absolute global errors $\lessapprox 1\%$.

\subsection{Convergence to random}
\label{sec:conj}
Finally, to conclude our analysis we present a conjecture that is suggested by our simulations and results (see \cite{psg_package} for experiments). Consider the (discrete) space of all properly bounded strategies $\mathcal{S}$ which determines the space $\mathcal{P}$ of all corresponding P-functions. The error $\epsilon_{ij} = \epsilon_{ij}^{n,N}$  between two strategies $S_i$ and $S_j$ defined in \ref{err} depends on $n$ and $N$. It is straightforward to check that $\epsilon_{ij}$ defines a metric on $\mathcal{P}$ for every value of $n$ and $N$. Hence we present the following.
\begin{con}
\label{conj}
For any properly bounded strategy $S$, excluding ${\rm KS}^0$, its P-function converges to the P-function of the random strategy as $N$ grows large or as $n/N$ grows small, in particular
\begin{align}
\label{eq:conjecture}
\lim_{N\to \infty}\epsilon_{S,\mathrm{RS}} &=  0 \;, \qquad {\rm for \;} n = N,
\\
\label{eq:conjecture1}
\lim_{{n}/{N}\to 0}\epsilon_{S,\mathrm{RS}} &=  0 \;, \qquad {\rm for \; } n < N.
\end{align}
\end{con}
To phrase \eqref{eq:conjecture} and \eqref{eq:conjecture1} differently, we also recall that a sequence of random variables $X_k$ is said to converge in distribution to $X$ if their associated cumulative distribution functions $F_k(x)$ converge to $F(x)$ for all $x$. In our case, where $X_k$ is associated to a strategy $S_k$, similar to \eqref{XRS}, we have for any fixed $a,n,$ and $N$,
\begin{equation}
F_k(w) = P(X_{k}\le w) = 1 - P^\text{min}_{S_k} (a,w+1).
\end{equation}
Then we may refine the conjecture to the statement that for any properly bounded strategy $S$ different from ${\rm KS}^0$, the associated random variable $S$ converges to $X_\text{RS}$ as $N$ grows large. Furthermore, even the ${\rm KS}^0$, for which the chances of winning the classic (constrained) problem were many orders of magnitude larger than the chances of winning using the random strategy, when considering the unconstrained problem treated is this paper, is only slightly ($\sim 25\%$) more efficient than the random strategy in the $a$-$w$ plane. 

\begin{rem}
Finally, we present informally an alternative formulation of the above conjecture, which does not rely on limits for $N $ or $n$. A nontrivial change in the algorithmic steps of a strategy $S$ (e.g. the value of $S$ on $B\times K$ for each prisoner $P_i\in P$) will move the point $S$ in $\mathcal{S}$. The closer two points $S$ and $S^\prime$ in $\mathcal{S}$, the smaller the differences (absolute value of the error or estimated efficiency) between the two strategies. Then we expect that there is a sequence of modifications of a strategy $S_k$ such that $S_k$ converges to the RS in distribution. Put differently, we can say that the random strategy acts as an \emph{attractor} in the space of all (properly bounded) strategies one can generate to solve the PSG. 
\end{rem}

\section{Future directions: rigged search games, storage strategies and memory}
\label{sec:6}
In this paper we discussed a multitude of inequivalent strategies to approach the unconstrained 100 prisoner problem, or PSG. Of course, any optimal retrieval process must clearly depend on the storage processes preceding it. Here, we analyzed the case of storage strategy in which $n$ keys are randomly distributed in $N$ boxes of capacity $0$ or $1$ (each box can contain at most one key) and show that all classes of strategies (except the key strategy with fixed initial conditions, $D=0$ and $n=N$) will achieve performances comparable, or converging to, the random strategy performances. However, as it turns out the key-strategy itself is not much more efficient than the random strategy, as it is easy to verify by computing the efficiency of a perfect God-strategy (see footnote $3$), $\eta \sim 5.6$. As expected, these results indicate that classical retrieval strategies for the unconstrained problem cannot improve by much the random strategy efficiency, since the storing process was fixed to be random.

On the other hand, the quantum (constrained) variant of the locker puzzle \cite{DB} has been shown to be trivial, as it allows for maximum winners with $100\%$ chances in all fair games for $a\ge \tfrac{\pi}{4}\sqrt{N}$. Based on recent discoveries \cite{memory1}, it would seem that neither the classical nor the quantum variant of this search game are fit to accurately describe memory retrieval processes in living organisms, but instead they constitute the two extreme mathematical abstractions of realistic memory retrieval processes. Natural organisms endowed with memory certainly do not store information in random fashion nor do they retrieve them randomly, and are most likely not constrained by quantum mechanical rules, at least at the collective level. Rather, the storage follows specific patterns, and consequently the same happens during retrieval processes. Furthermore, in various circumstances, memory retrieval processes are facilitated by external (or internal) cues (see for instance \cite{memory2}), which allow for the right neuronal connection to be turned on, producing memory recall.

The goal of this paper was to pave the way for the mathematically rigorous but biology-driven analysis of memory in living organism, animate or inanimate. Of course, to describe accurately biological memory processes one would need to consider further extensions or modifications of the PSG. For instance, crucial to resolve is the equivalent problem posed on storing strategies: given a fixed retrieval strategy, what is the storing strategy which optimizes the probability of (partial) information retrieval? Alternatively, is there a storing strategy which optimizes the probability of (partial) information retrieval, regardless of the retrieval strategy?

To answer these questions in general terms, it will be necessary to consider boxes with capacity $l>1$, i.e. each box can contain at most $l$ keys: this simple change will destroy the key-box cycle-structure created by randomly distributing keys in boxes all with exact capacity $1$, and create branching as well as overlapping cycles structures. Equivalently intriguing is the generalization whereby each prisoner $P_i$ looks for her key $i$, but also for other distinct keys $j,k,\dots$ which also open her cell door.\par
One would certainly expect these extended PSGs to possess novel, less-smooth P-functions, obviously related to the chosen storing strategy, which may be directly confronted to efficiency or performance of biological mnemonic processes.

\subsection*{Acknowledgements}
IL acknowledges two interviewers from Alibaba who accidentally made him aware of the problem. SMS acknowledges funding from the DST-INSPIRE Faculty Fellowship (DST/INSPIRE/ 04/2018/002664) by DST India.  TAW The was partially supported by NSF grant DMS-2212924.

\appendix

\section{Unbounded (pure) random strategy}
\label{ref:appA}

In this appendix, we present the analytic P-function formula for the unbounded random strategy. This variant of the strategy (for $n=N$) presented in \ref{subsec:3.1} allows every prisoner to choose and open a box that was already opened. This could happen, for instance, if the prisoner suffers from short-term memory loss or if the amount of boxes she opened is so large that she would not remember them all.
At each box selection a prisoner has 1 chance of getting the right box with the right key, and $N-1$ chances of getting the wrong box. At the next selection, she again has $1$ chance of finding the right box, and again $N-1$ of failing. Instead of summing all chances that the prisoner gets her key in the first attempt, or the second, or multiple attempts, we can compute directly the chances that a prisoner does not obtain her key in $a$ attempts,
\begin{equation}
\Big(\frac{N-1}{N}\Big)^a.
\end{equation}
Consequently, the probability that a prisoner does get her key within $a$ attempts is easily obtained by subtracting the above number from the unity. If we consider  $N$ players and impose that exactly $w$ must be winners, we obtain the P-function for the pure random strategy:
\begin{equation}
P_{\rm pure R}(a,w)={{N}\choose{w}} \Big[1-\Big(\frac{N-1}{N}\Big)^a\Big]^w \,\Big[\Big(\frac{N-1}{N}\Big)^a\Big]^{N-w}.
\end{equation}
In figure \ref{fig:pure-random} we present the P-function for $N=100$. The minimum-winner P-function for the pure random strategy is everywhere smaller than (or equal to) the properly bounded random-strategy P-function, as expected from the discussion in section \ref{subsec:2.1}, but this inequality does not hold for the exact-winner P-function in the whole $a$-$w$ plane.\\
This is also confirmed by the estimation of the efficiency for this unbounded strategy, $\eta= 0.66$, the lowest efficiency value encountered throughout the analysis. 
\setcounter{equation}{0}
\begin{figure}[h!]
\begin{center}
\includegraphics[width=0.33\textwidth]{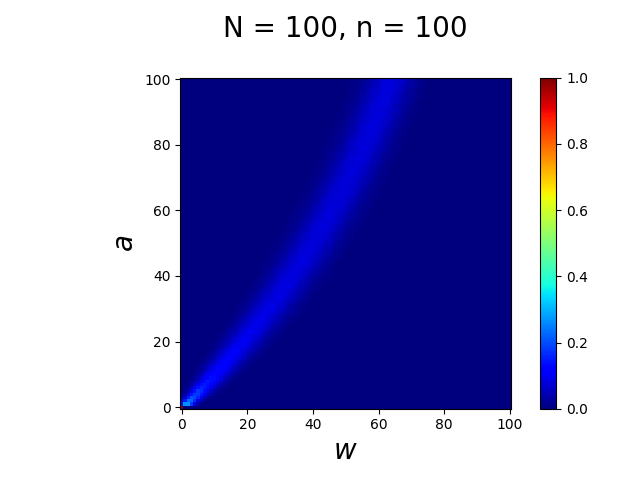} 
\includegraphics[width=0.33\textwidth]{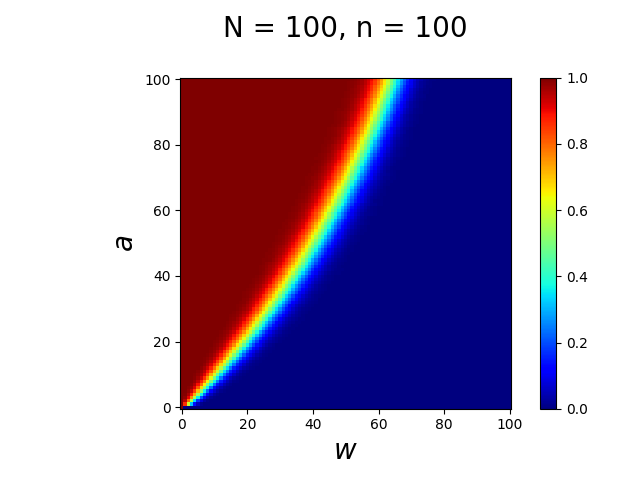} 
\end{center}
\caption{The P-functions for the pure-random strategy, the only example of unbounded strategy presented in this paper.}
\label{fig:pure-random}
\end{figure}

\section{Proofs of estimates}
\label{AppB}
In this appendix, we prove the formulas presented in the main text for the P-functions of the Key and Box strategies.
\begin{lem}
\label{lemma:omega}
The probability for all $N$ players to find their key within $a$ attempts, where $1\le a \le N$ is 
\begin{equation}
P_{\mathrm{KS}^0}(a,N)=1 - \sum_{k = a + 1}^N\frac{\Omega^N_k}{N!},
\end{equation}
where $\Omega^N_k$ is defined by \eqref{eq:gen_class_formula} and \eqref{eq:recursive_Omega} below.
\end{lem}

\begin{proof}
Call $r=\lfloor N/l\rfloor$, so $r\ge 1$ always and in \eqref{eq:class_count} $r=1$. By definition, for $r\ge 1$:
\begin{equation}
\label{eq:def_Omega}
\Omega^N_l=\sum_{i=1}^r \Omega^N_{\alpha_l=i}, \qquad \Omega^N_{\alpha_l=0}=N!-\Omega^N_l
\end{equation}
where $\Omega_{\alpha_l=i}$ is the number of permutations which contain exactly  $i$ many $l$-cycles. Now, a known formula counts the number of permutations which have a specific cycle decomposition or partition of cycles $p=(\alpha_1,\dots,\alpha_N)$, equation \eqref{eq:perm_cycle_count}.
This means that if all partitions of $N$ containing a certain number of $l$-cycles are known, a repeated use of the \eqref{eq:def_Omega} gives us $\Omega^N_l$. However, the number of partitions grows rapidly with $N$ so the brute-force approach is not feasible.  To solve this problem, we proved an identity which counts the number of permutations containing a precise number $k$ of $l$-cycle, i.e. $\alpha_l=k$, in terms of the number of permutations containing exactly $k-k_1\ge 0$ and $k_1$ $l$-cycles respectively:  
\begin{equation}
\label{eq:recursive_Omega}
\Omega^N_{\alpha_l=k}=\frac{{N \choose k_1\cdot l}}{{k\choose k_1}}\Omega^{k_1\cdot l}_{\alpha_l=k_1}\,\Omega^{N-k_1\cdot l}_{\alpha_l=k-k_1}
\end{equation}
valid for all values of $N$, $k>0$ and for all $k_1$ such that $k\ge k_1\ge 1$ (in fact we obtain a trivial identity for $k_1=0$). In the above, the binomial coefficients count all the ways $k_1\cdot l$ elements can be chosen from N elements and the number of ways $k_1$ $l$-uples can be chosen among $k$. Note also that \eqref{eq:class_count} is a special case of \eqref{eq:recursive_Omega}.
Of course, some values of the factors $\Omega$ can be easily computed by using \eqref{eq:perm_cycle_count} or simple definitions, such as:
\begin{align}
\Omega^{k_1\cdot l}_{\alpha_l=k_1}=\frac{k_1\cdot l!}{l^{k_1} k_1!}, \qquad \Omega^l_{\alpha_l=1}=(l-1)!, \qquad \Omega^{N-l}_{\alpha_l=0}=(N-l)! \quad\text{for } l>N-l,
\label{eq:simple_rel}
\end{align}
where the last formula extends the result \eqref{eq:recursive_Omega} to the case $k=0$, but only whenever it is not possible to create an $l$-cycle among $N-l$ elements, simply because there are not enough elements available. In all other circumstances, $\Omega^N_{\alpha_l}=0$ can only be computed indirectly from the second of \eqref{eq:def_Omega}. It is a simple bookkeeping exercise to show that the above formula \eqref{eq:class_count} can be easily obtained from  \eqref{eq:recursive_Omega} by making use of the last two equations in \eqref{eq:simple_rel}.

We can now derive the generalization of \eqref{eq:class_count} to the cases $l<N/2$, i.e. $r\ge 2$: from both \eqref{eq:def_Omega}, using \eqref{eq:recursive_Omega} and fixing $k_1=1$ (to account for the permutations of $N$ containing only $1$ $l$-cycle),  we get
\begin{align}
\Omega^N_l&={N\choose l} \Omega^{l}_{\alpha_l=1}\sum_{i=1}^r\, \frac{\Omega^{N-l}_{\alpha_l=i-1}}{{i \choose 1}}
\nonumber\\
&=\frac{N!}{l\cdot (N-l)!}\Big(\Omega^{N-l}_{\alpha_l=0} + \sum_{i=2}^r \frac{\Omega^{N-l}_{\alpha_l=i-1}}{i}\Big)
\nonumber\\
&=\frac{N!}{l\cdot (N-l)!}\Big[(N-l)!-\sum_{i=1}^{r-1}\Omega^{N-l}_{\alpha_l=i}+\sum_{i=2}^r \frac{\Omega^{N-l}_{\alpha_l=i-1}}{i}\Big]
\nonumber\\
\label{eq:gen_class_formula}
&=\frac{N!}{l\cdot (N-l)!}\Big[(N-l)!-\sum_{i=1}^{r-1}\frac{i}{i+1} \Omega^{N-l}_{\alpha_l=i}\Big].
\end{align} 
Clearly, if $r=1$ the above formula reduces to \eqref{eq:class_count}, whereas if $r> 1$, namely $l< N$, the recursive use of \eqref{eq:recursive_Omega} is convenient to obtain $\Omega^{N-l}_{\alpha_l=i}$.
\end{proof}
For the Box strategy, we first show that the P-function satisfies the same reflection-symmetry property of the random strategy.
\begin{lem}
\label{lem:box}
 For any $1\le a \le N$, $0\le w \le N$, and $I$ coprime to $N$, the box-strategy exact-winner P-function satisfies:
\begin{equation} 
\label{eq:symm_box}
P_\mathrm{BS}(a,w)=P_\mathrm{BS}(N-a,N-w).
\end{equation}
\end{lem}
\begin{proof}
If $I$ is coprime to $N$, the box strategy is properly bounded. Hence, any prisoner $P_i$, given $N$ attempts would open all $N$ boxes, in the order $(i+D){\rm mod\;}N,(i+D+I){\rm mod\;}N,(i+D+2I){\rm mod\;}N,\dots,(i+D+(N-1)I){\rm mod\;}N$, with $D$ arbitrary offset. Now, for every permutation such that $w$ players $P_j$ are winners in $a$ attempts, i.e. find their key in one of the boxes $(j+D){\rm mod\;}N,(j+D+I){\rm mod\;}N,,\dots,(j+D+(a-1)I){\rm mod\;}N$, there are $N-w$ players $P_l$ which did not find their key in $a$ attempts, but would have certainly found their key in one of the unopened boxes $(l+D+aI){\rm mod\;}N,(l+D+(a+1)I){\rm mod\;}N,\dots,(l+D+(N-1)I){\rm mod\;}N$. Hence, the players $P_l$ would have won, i.e. would have found their key, if they used the box strategy with offset $D^\prime=D+aI$, same increment $I$ and also had at their disposal $N-a$ attempts. We then arrive at the equality $P^D_{\rm BS}(a,w)=P^{D^\prime}_{\rm BS}(N-a,N-w)$. Finally, because the P-functions of the box strategy are independent of $D$, as we shall argue at the end of the section, the identity \eqref{eq:symm_box} follows.   
\end{proof}
\par It should be quite clear why the same property will not hold for the key-strategy. In that case, in fact, no player can know, before playing, which box(es) she will open for any given attempt $a>1$. If $w$ players find their key in $a$ attempts, the remaining $N-w$ players could not now a priori the box to start from, hence would not have certainly found their cycle, and hence their key.

\begin{proof}[\bf Proof of Proposition \ref{boxprob}]
\label{proof:prop3.4}
We first consider $a=1$. This case corresponds to the ``probl\`{e}me des rencontres'': we want to evaluate the number of permutations of $N$ such that no players will find their key in its box. 
If $k$ elements are fixed in a position, there are exactly $(N-k)!$ permutations of the remaining elements in the remaining positions. Furthermore,  there are ${N}\choose{D}$ ways to select those elements. 
Hence, from the inclusion-exclusion principle,\footnote{This principle was in fact firstly used to find a solution to the ``probl\`{e}me des rencontres''  \cite{riordan}} we get:
\begin{equation}
\label{eq:Dn0}
D_{N,0}=\sum_{k=0}^N (-1)^k {N\choose k}\; (N-k)! = N!\sum_{k=0}^N \frac{(-1)^k}{k!},
\end{equation}
where the index $0$ indicates that all elements possess the property $a$, or no element possesses the property $\bar{a}$. This formula is well-known and counts the number of ``derangements'' $D_{N,0}$, permutations with no trivial cycles. It can easily be generalized: if we are interested only in $N-w$ elements satisfying the property $a$, or equivalently $w$ elements satisfying $\bar{a}$ (and hence permutations with $w$ trivial cycles), we then obtain:
\begin{equation}
\label{eq:Dnw}
D_{N,w}={N\choose w} D_{N-w,0}={N\choose w}\,(N-w)!\sum_{k=0}^{N-w}\frac{(-1)^k}{k!}=  {N\choose w}\,!(N-w),
\end{equation}
where the symbol $!p$ indicates the derangements and the sum has been truncated to reflect the fact that now we want at most $(N-w)$-uples of elements to satisfy the condition $a$ at the same time. We have also multiplied by a binomial coefficient, which simply counts the $w$-uples of elements falling in their positions, also called the ``hits'', within $N$ total elements. Now it is straightforward to obtain the explicit formula for the probability for $w$ players to find their key at the first attempt 
\begin{align}
P_\mathrm{BS} (1,w) &= \frac{D_{N,w}}{N!}={N \choose w}\frac{(N-w)!}{N!}\sum_{k=0}^{N-w}\frac{(-1)^k}{k!}
\nonumber\\
\label{eq:derange1}
&= \frac{1}{w!}\sum_{k=0}^{N-w}\frac{(-1)^k}{k!}
\end{align}
as desired. Note that, from \eqref{eq:symm_box} we get:
\begin{equation}
P^N_{\rm BS}(N-1,N-w)=  \frac{1}{(N-w)!} \sum_{k=0}^{w}\frac{(-1)^k}{k!}\;.
\end{equation}

It is also instructive to derive the above formula from the generating polynomial of derangements \cite{riordan},
\begin{equation}
\sum_{k=0}^N \frac{N!}{k!} (x-1)^k.
\end{equation}
The coefficient of the term $x^w$ enumerates the number of permutations such that $w$ players have hits, i.e. they find their key in their-number box in one attempt. By expanding the polynomial, using the fact that the diagonals of the Pascal's triangle are given by binomials \footnote{To be more precise, consider the $i^{\rm th}$ elements on the left (right) edge of the Pascal's triangle. The elements $j^{\rm th}$ lying on the diagonal starting from the element, then reaching the lower right (left) successive elements is simply given by ${i}\choose{j-1}$}, we get:
\begin{equation}
\label{eq:derange2}
P^N_{\rm BS}(1,w)=\frac{1}{N!}\sum_{k=w}^N (-1)^{k-w}\frac{N!}{k!} {k\choose w},
\end{equation}
which is equal to \eqref{eq:derange1} after a few simplifications and a shift in the range of $k$.

Next, we consider the case $a=2$. This case is complementary to the ``probl\`{e}me des m\'{e}nages'' \cite{2disc} formulated as: 
{What is the number of permutations such that no element $i$ is found in position $(i,i+1)\bmod N$ (meaning that the last element $N$ cannot be in position $N$ or $1$)?}

We can derive the formula directly from the known generating polynomial \cite{riordan}:
\begin{equation}
\sum_{k=0}^N \frac{2N}{2N-k}{2N-k \choose k} (N-k)! (x-1)^k.
\end{equation}
As before, the coefficient $C_{N,w}$ of the term $x^w$ enumerates the number of permutations such that $w$ players have hits, i.e. find their key in their-number box or in the following one (clearly no overlap is allowed, each box contains 
one and only one key, so no two players can find their key in the same box). The probability then reads
\begin{equation}
P^N_{\rm BS}(2,w)=\frac{C_{N,w}}{N!}=\frac{1}{N!}\sum_{k=w}^N (-1)^{k-w} \,\frac{2N}{2N-k}{2N-k \choose k}(N-k)! {k \choose w}
\end{equation}
as desired. It is easy to see the similarities with \eqref{eq:derange2}.
\end{proof}
\newpage
\section{Hybrid strategies P-functions plots}
In this appendix we present the plots for the P-function of the hybrid strategies discussed in section \ref{sec:4}\\
\begin{figure}[h!]
\begin{center}
\includegraphics[width=0.33\textwidth]{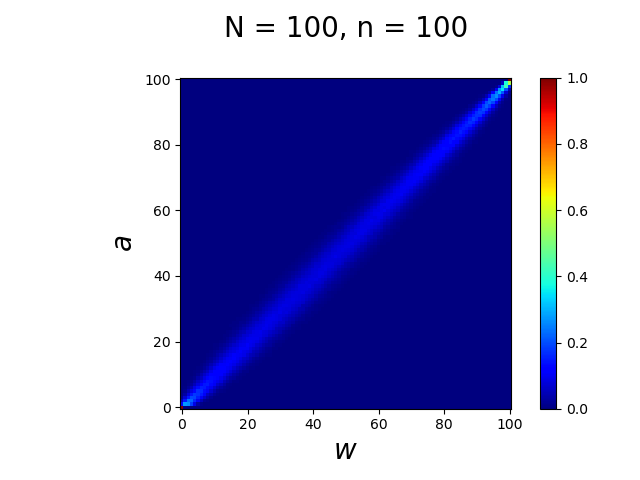} 
\includegraphics[width=0.33\textwidth]{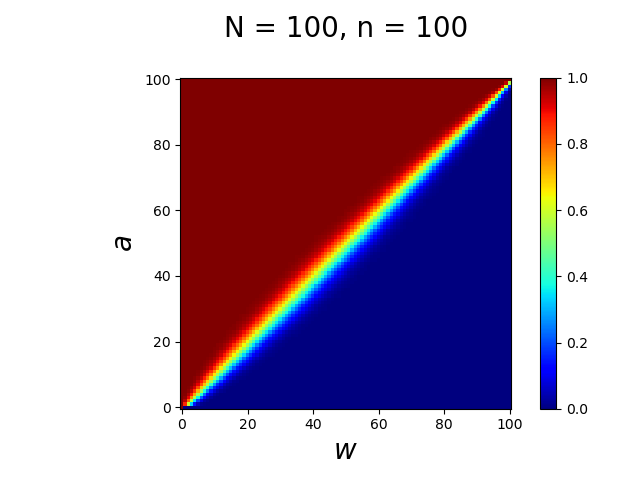} 
\end{center}
\caption{The key strategy P-function with $D\neq 0$, made properly bounded by the use of the random strategy as an escape, $E=R$. Note that the graph does not change if the value of the offset is different, or not constant for different prisoners, and also if the escape strategy is not the random but the box strategy $E=B$.}
\label{fig:key+rand}
\begin{center}
\includegraphics[width=0.33\textwidth]{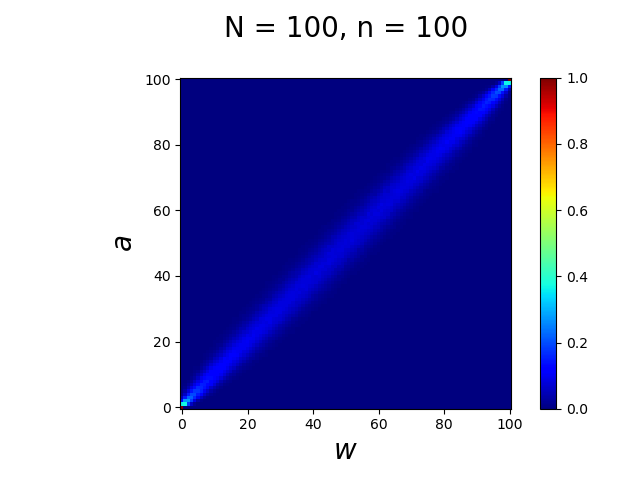} 
\includegraphics[width=0.33\textwidth]{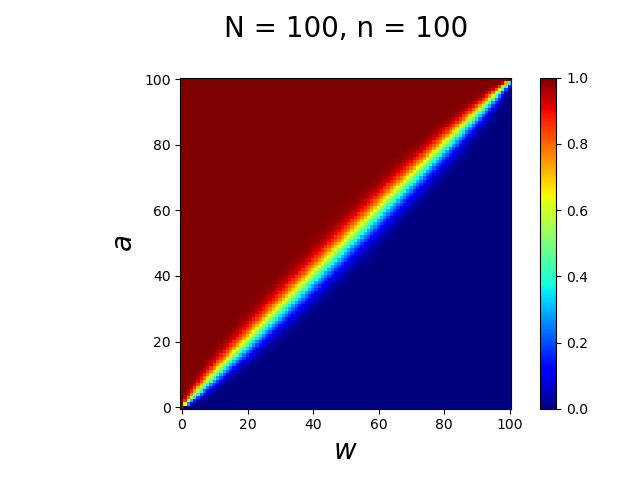} 
\end{center}
\caption{The box strategy P-function for ${I}=5$ (not coprime with $N=100$) and $D=0$, made properly bounded by the use of the random strategy as an escape. Note that the graph does not change for a different choice of the offsets or if the escape strategy chosen is the box strategy with a different increment.}
\label{fig:box+rand}
\end{figure}
\begin{figure}[h!]
\begin{center}
\includegraphics[width=0.33\textwidth]{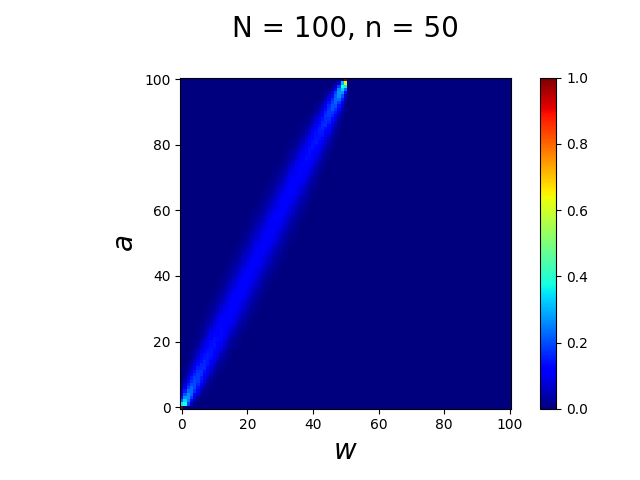} 
\includegraphics[width=0.33\textwidth]{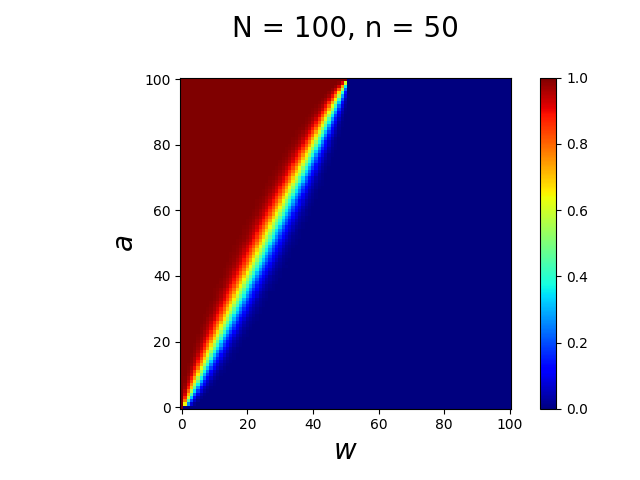}
\end{center}
\caption{The properly-bounded random-strategy P-functions for $n<N$}
\label{fig:random_n}
\end{figure}

\begin{figure}[t!]
\begin{center}
\includegraphics[width=0.33\textwidth]{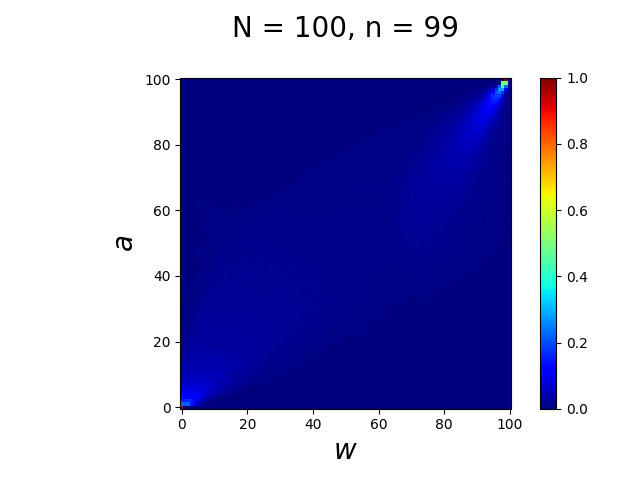} 
\includegraphics[width=0.33\textwidth]{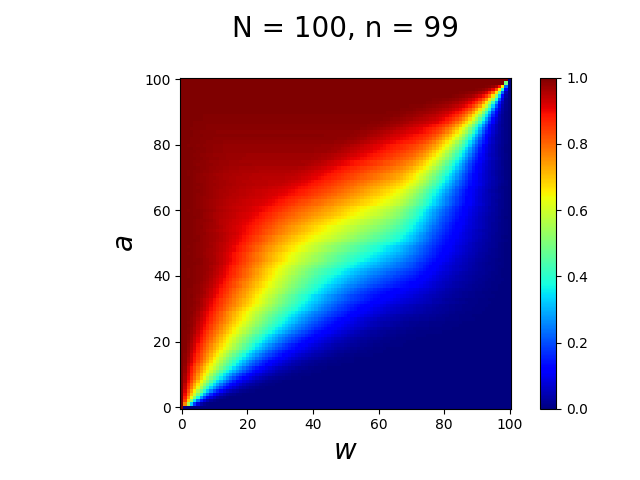} 
\end{center}
\caption{P-functions of the key strategy with ${D}=0$ and escape $E=R$.}
\label{fig:key+random_offset0}
\begin{center}
\includegraphics[width=0.33\textwidth]{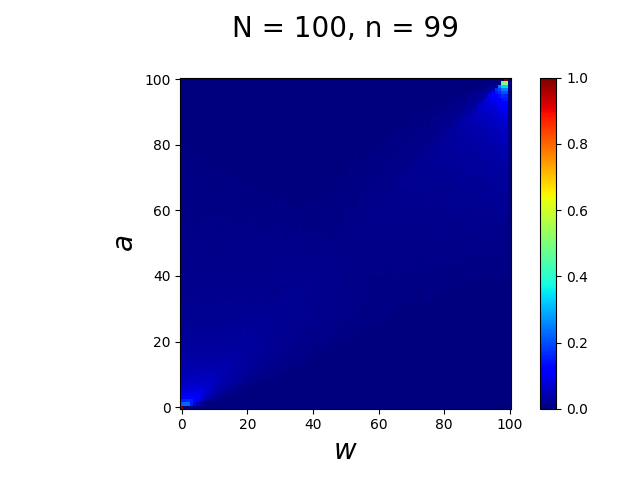} 
\includegraphics[width=0.33\textwidth]{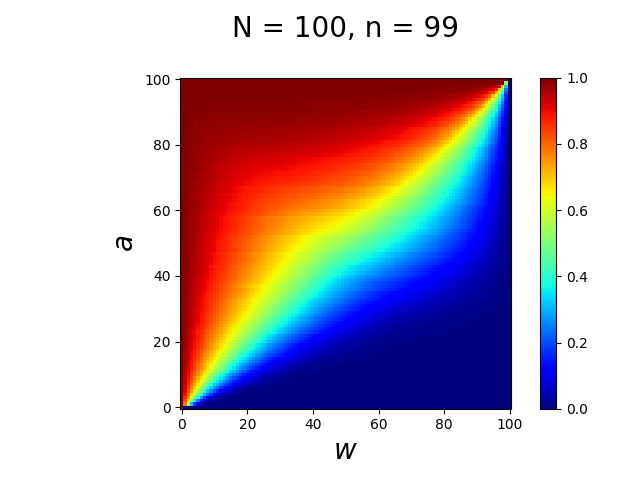} 
\end{center}
\caption{P-functions of the key strategy with ${D}=0$ and escape $E=B$.}
\label{fig:key+box_offset0}
\begin{center}
\includegraphics[width=0.33\textwidth]{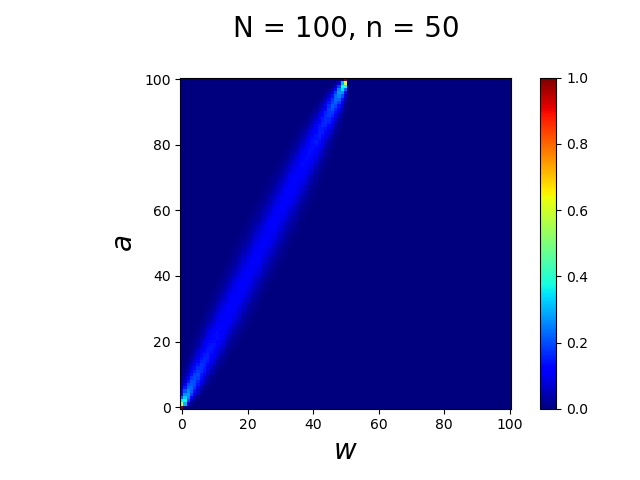} 
\includegraphics[width=0.33\textwidth]{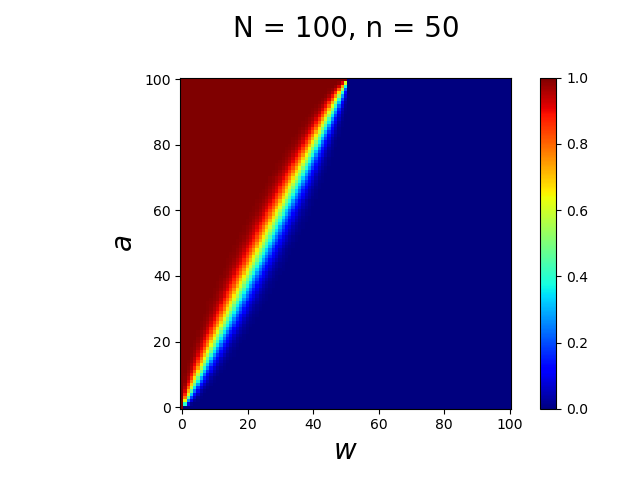} 
\end{center}
\caption{The unbounded box-strategy P-function with ${D}=2$ and ${I}=5$ (not coprime to $N=100$), made properly bounded by the use of the random strategy as an escape, $E=R$. Note that the graph does not change for a different choice of the offsets or escape route.}
\label{fig:box+random}
\vspace{0.3cm}
\begin{center}
\includegraphics[width=0.33\textwidth]{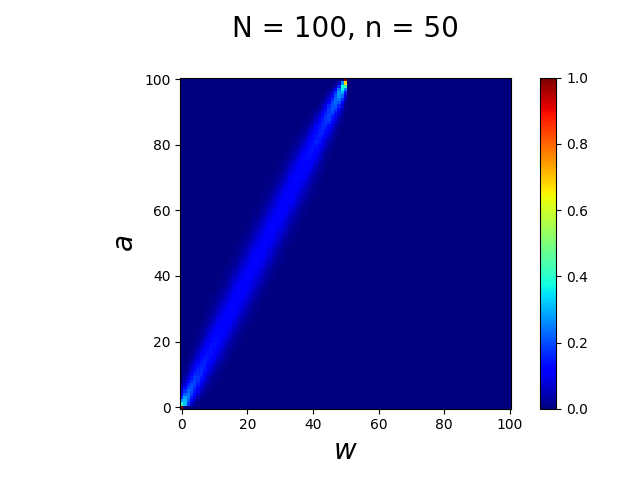} 
\includegraphics[width=0.33\textwidth]{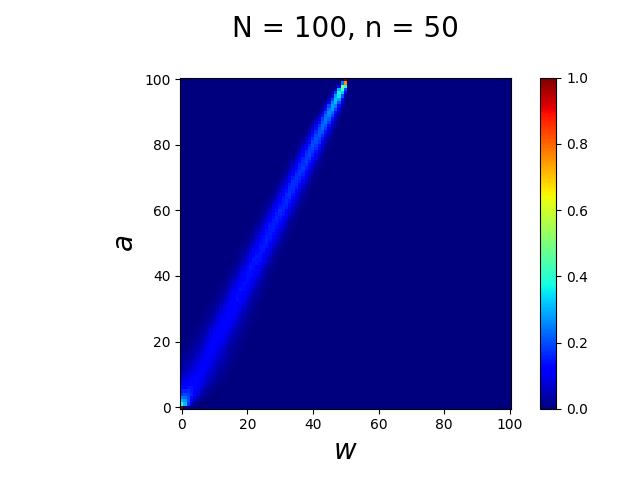} 
\end{center}
\caption{The exact winners P-function for the hybrid strategy, KS with $D=0$, made properly bounded by either the random (top) or sequential (bottom) escape route. Though almost identical, note that (non-statistical) differences for $a,w<20$, showing the two strategies are in fact different.}
\label{fig:2hybr}
\end{figure}

\begin{figure}[t!]
\begin{center}
\includegraphics[width=0.33\textwidth]{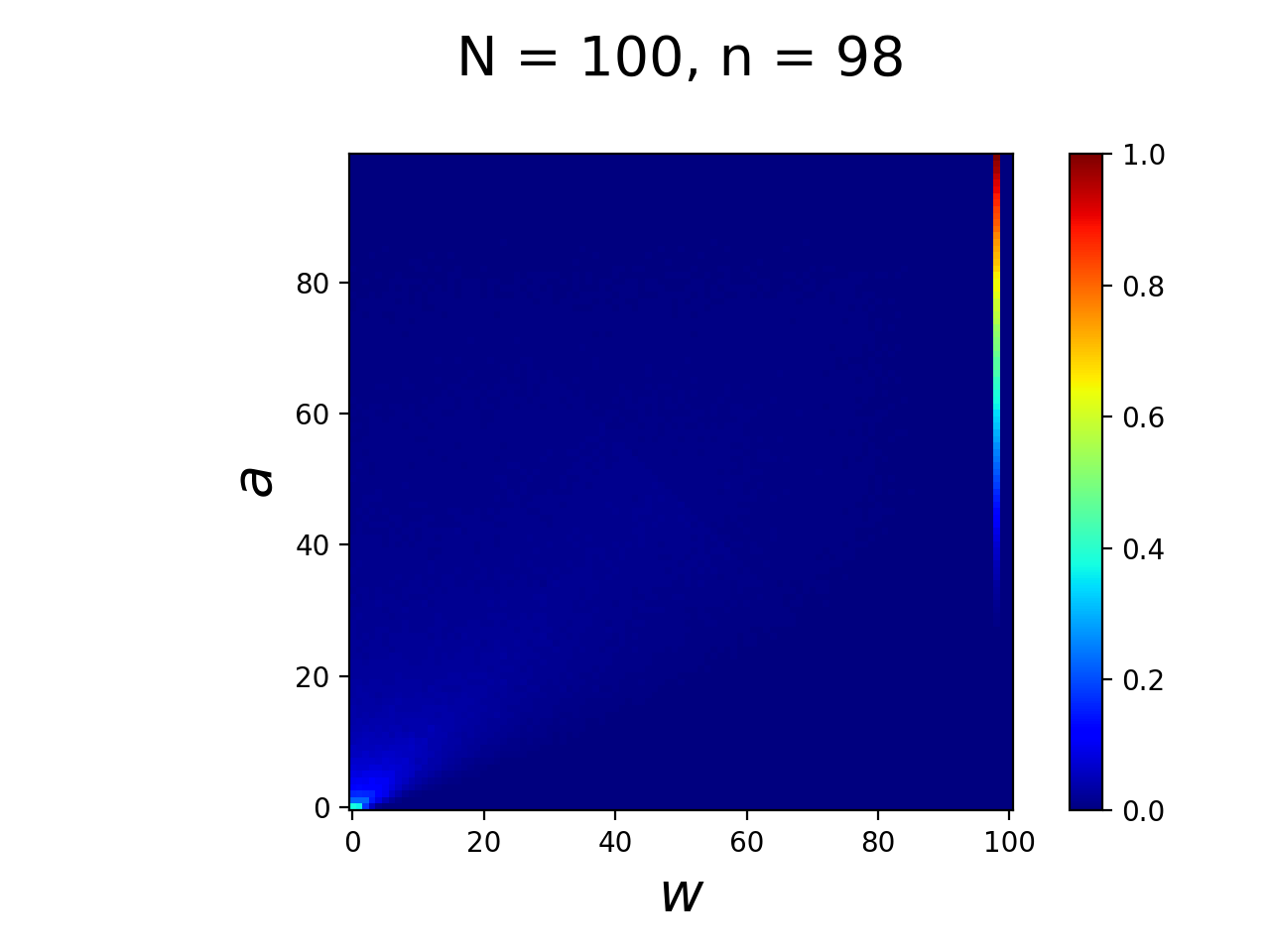} 
\includegraphics[width=0.33\textwidth]{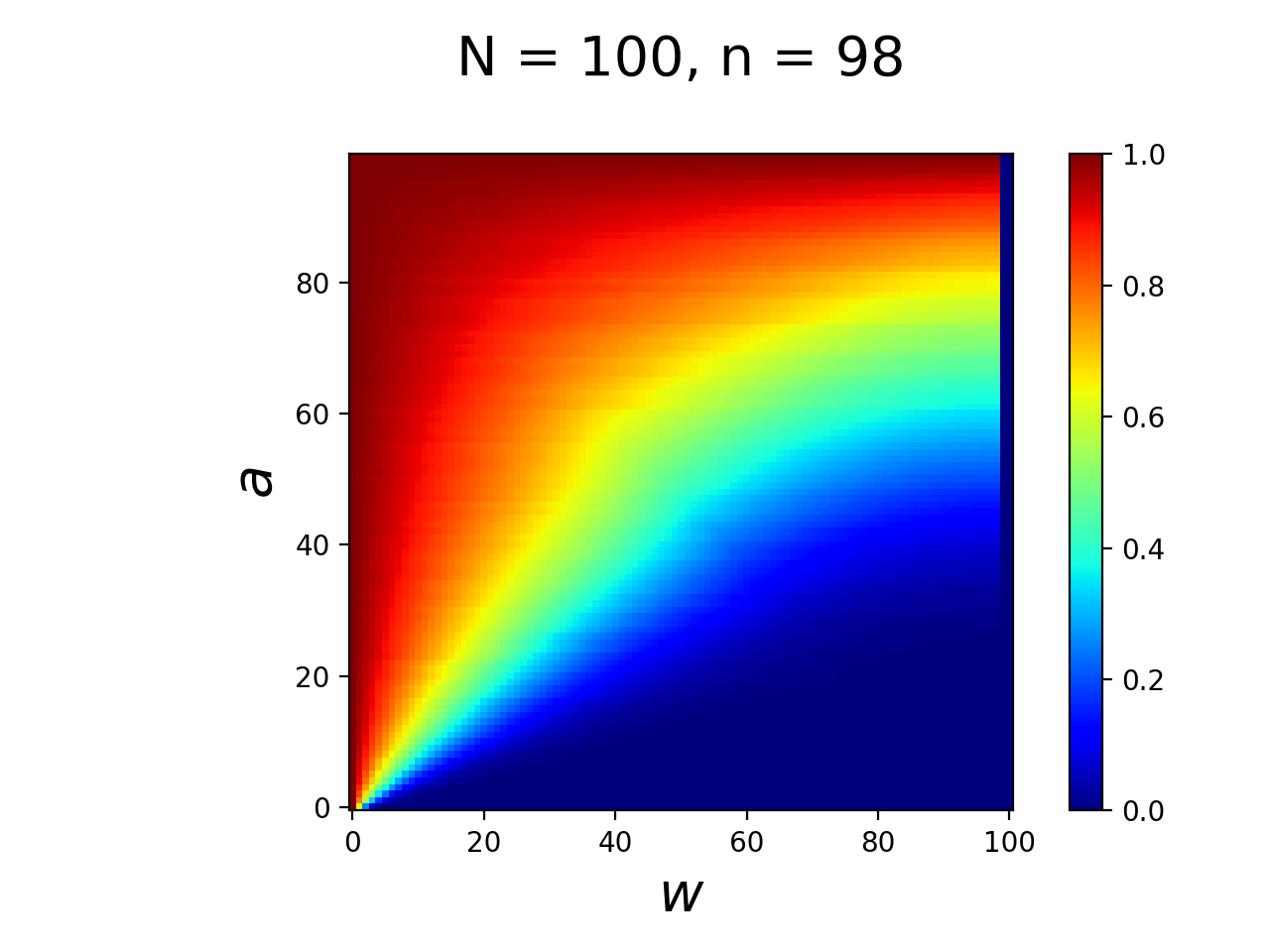} 
\end{center}
\caption{The ADI strategy, for $n=98$.}
\label{fig:adi98}
\begin{center}
\includegraphics[width=0.33\textwidth]{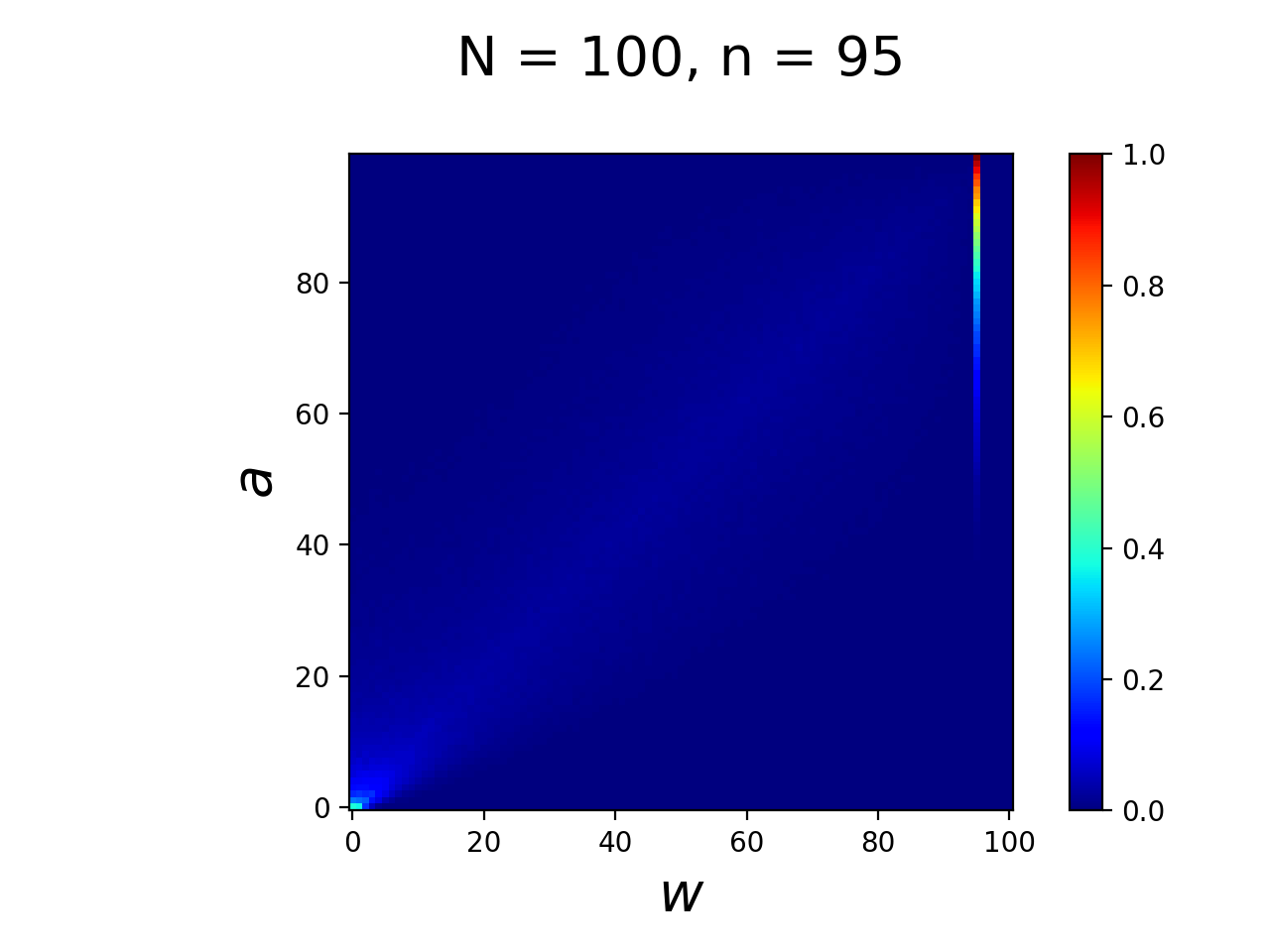} 
\includegraphics[width=0.33\textwidth]{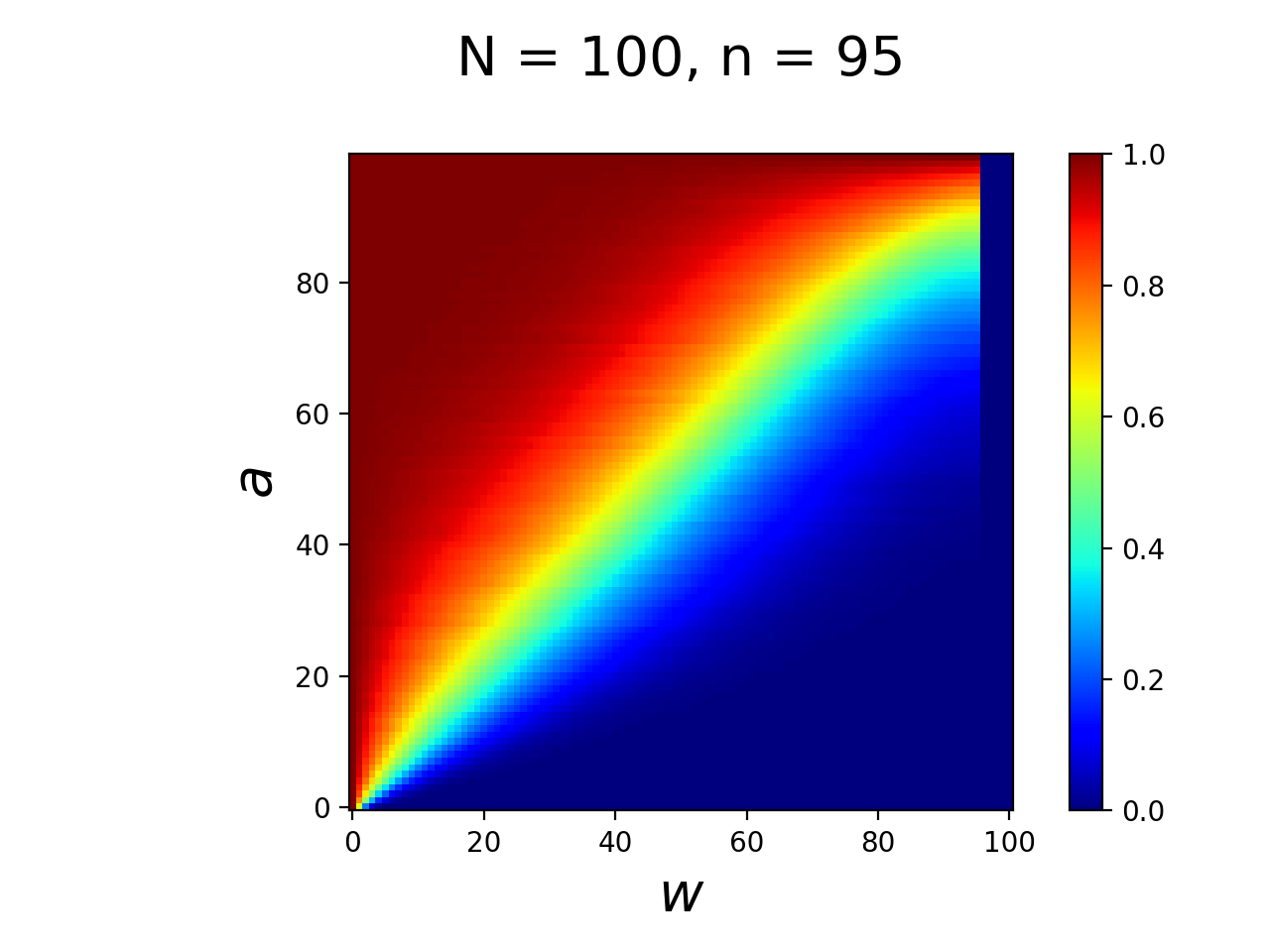} 
\end{center}
\caption{The ADI strategy, for $n=95$.}
\label{fig:adi}
\begin{center}
\includegraphics[width=0.33\textwidth]{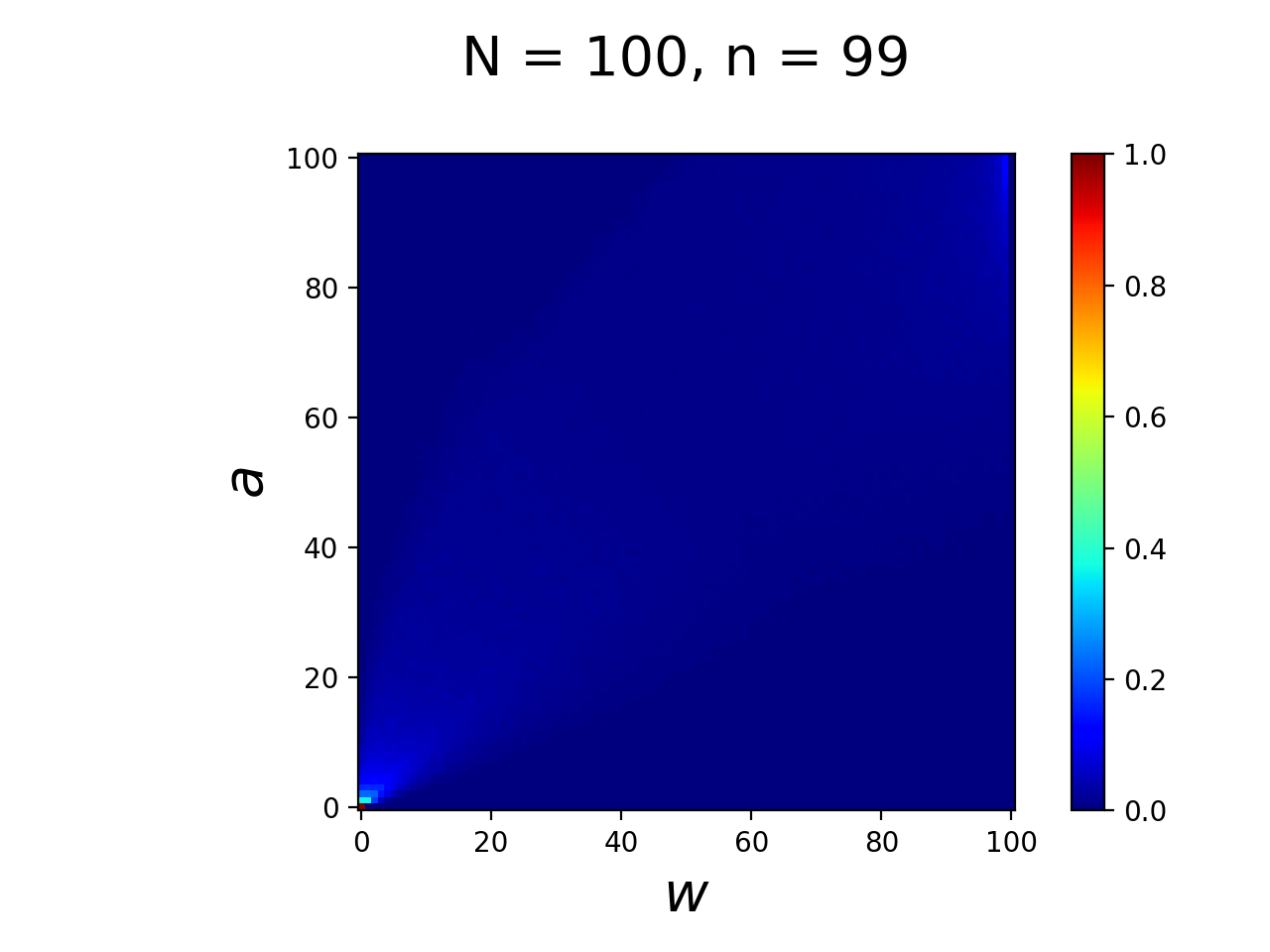} 
\includegraphics[width=0.33\textwidth]{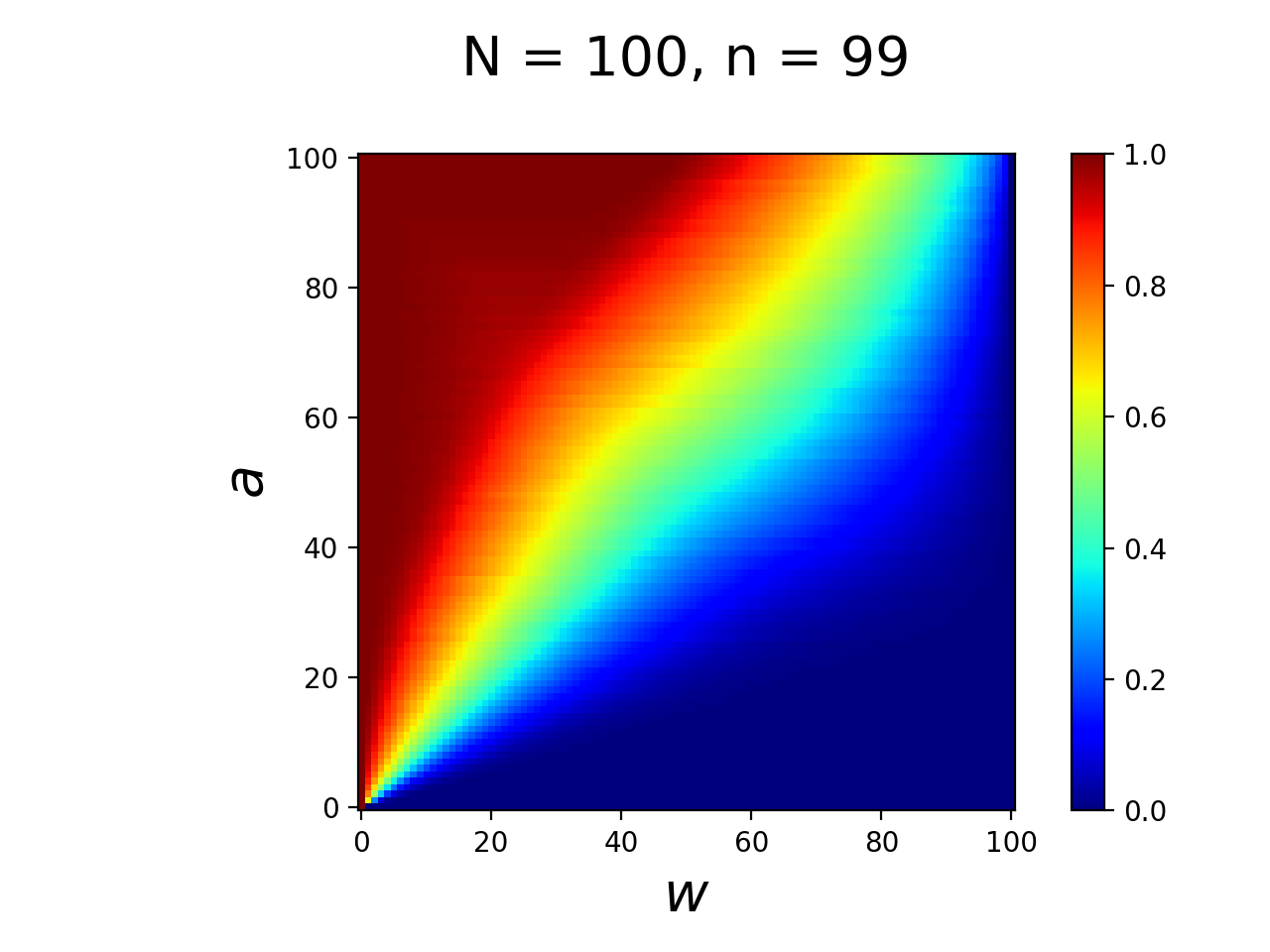} 
\end{center}
\caption{The Goyal-Saks strategy, the only example of a (not properly) bounded strategy presented in this paper.}
\label{fig:Goyal-Saks}
\end{figure}

\clearpage
\section{Error $\varepsilon_{ij}$}
 \begin{figure}[h]
\begin{center}
\includegraphics[width=0.45\textwidth]{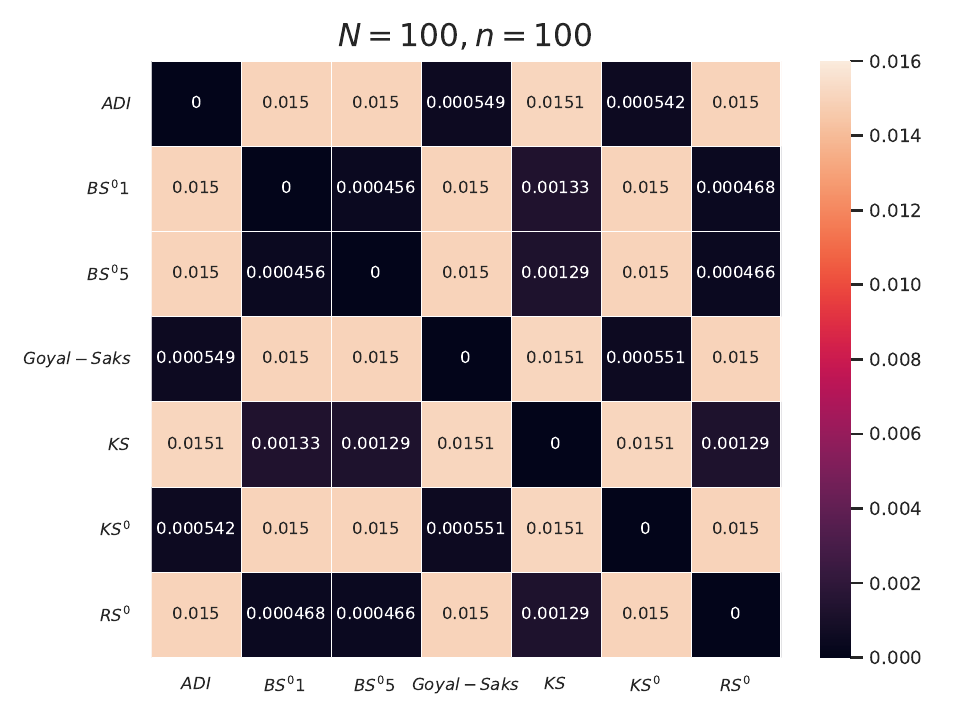} 
 \includegraphics[width=0.45\textwidth]{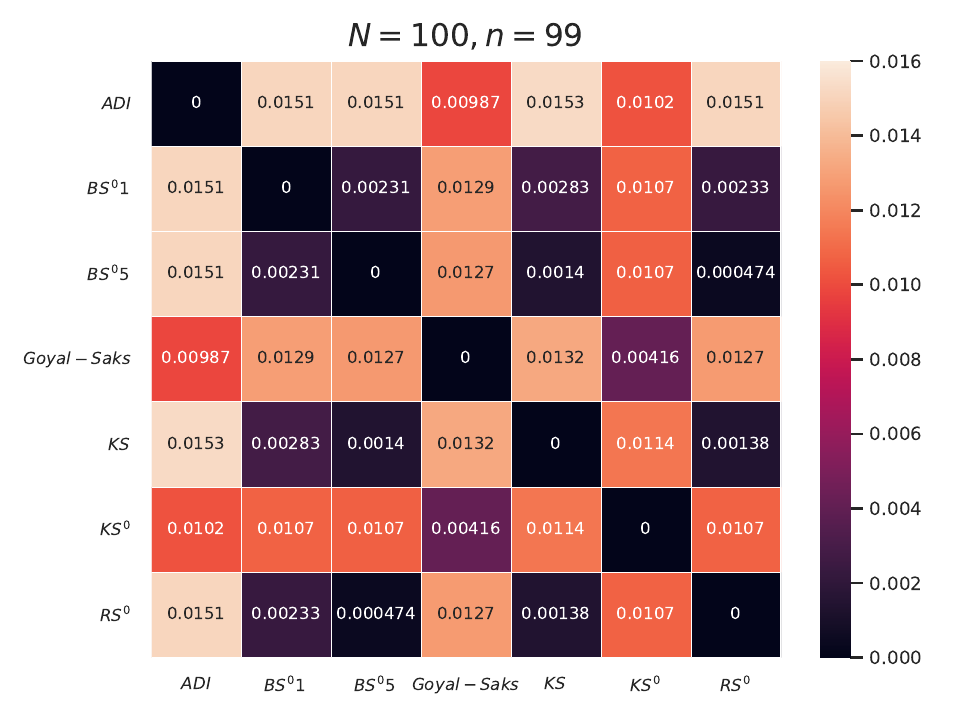}
\end{center} 
\begin{center}
\includegraphics[width=0.45\textwidth]{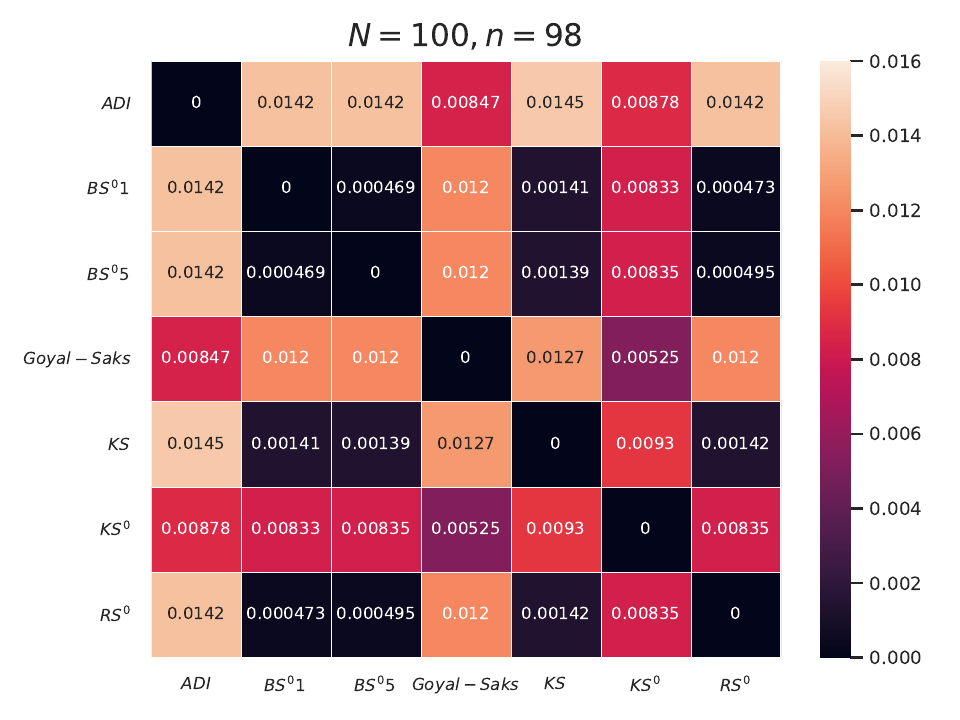} 
\includegraphics[width=0.45\textwidth]{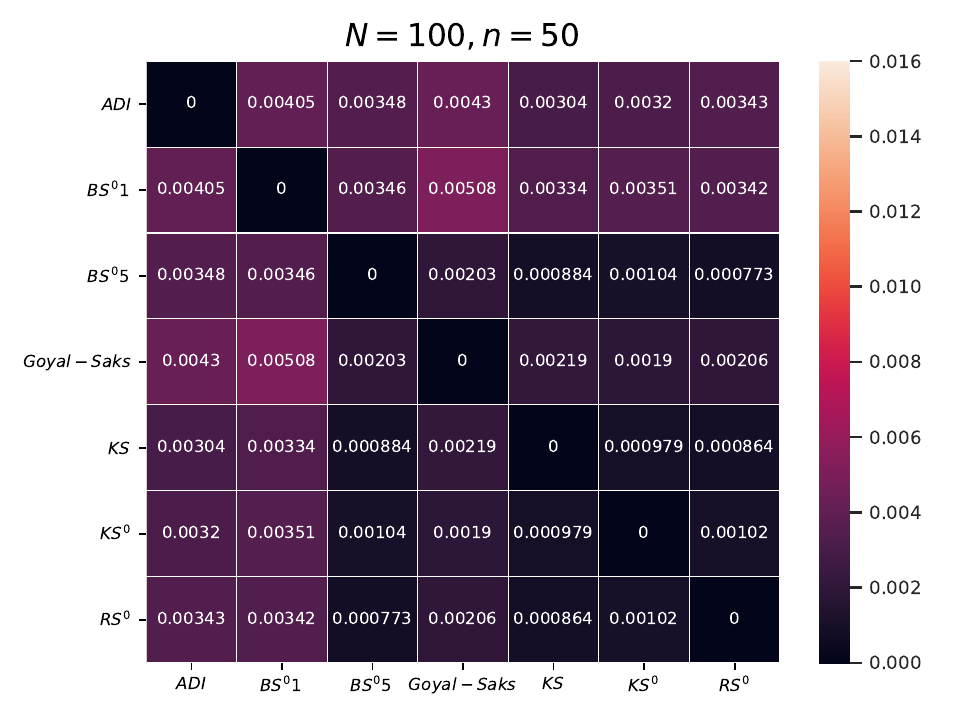}
\end{center}
\caption{Errors heatmap for various values of $n$. For each $n$, and for each pair of strategies, the absolute error is calculated using \eqref{err}.}
\label{fig:errors}
\end{figure}
\clearpage
\providecommand{\href}[2]{#2}

\end{document}